\documentclass[11pt,reqno]{amsart}

\usepackage{amsmath, amssymb ,amsthm,float}
\usepackage[a4paper,hmarginratio={1:1}]{geometry}

\usepackage[english]{babel}
\usepackage{tikz}

\usetikzlibrary{plotmarks}

\newcommand{\scal}[2]{ \langle #1 , #2 \rangle}
\newcommand{\ct}{\mathrm{ct}} 
\renewcommand{\Re}{\mathrm{Re}}

\newcommand{\NN}{\mathbb N}
\newcommand{\ZZ}{\mathbb Z}
\newcommand{\RR}{\mathbb R}
\newcommand{\CC}{\mathbb C}

\numberwithin{equation}{section}

\newtheorem{theorem}{Theorem}[section]
\newtheorem{lemma}[theorem]{Lemma}

\renewcommand{\Im}{\mathrm{Im}}
\renewcommand{\Re}{\mathrm{Re}}

\theoremstyle{remark}
\newtheorem*{remark}{Remark}

\theoremstyle{definition}

\begin{document}
\title[Eigenvalue asymptotics for elastic crack problems]{Eigenvalue asymptotics for an elastic strip and an  elastic plate with a crack}

\author{Andr\'e H\"anel}
\address{Andr\'e H\"anel, Institut  f\"ur Analysis, Leibniz Universit\"at Hannover, Welfengarten 1,  D-30167 Hannover}
\email{andre.haenel@math.uni-hannover.de}

\author{Timo Weidl}
\address{Timo Weidl, Institut f\"ur Analysis, Dynamik und Modellierung, Universit\"at Stuttgart, Pfaffenwaldring 57, D-70569 Stuttgart}
\email{weidl@mathematik.uni-stuttgart.de} 

\begin{abstract}
We consider the elasticity operator with zero Poisson's ratio  on an infinite    strip and an  infinite plate with  a horizontal crack. We prove an asymptotic formula for the  distance of the 
embedded eigenvalues to some spectral threshold of the operator as the crack becomes small.
% {2010 Mathematics Subject Classification:}  Primary: 35P20; Secondary: 74R10, 35P05, 74K20. \medskip
% {Keywords:} Elasticity operator,  elastic crack problem,  Dirichlet-to-Neumann operator, spectral asymptotics.
\end{abstract}

\maketitle

\section{Introduction}
In the present article we consider an elastic strip and an elastic plate with a horizontal crack. We are interested in the existence of trapped modes and their asymptotic behaviour
as the crack shrinks to a point. Mathematically a trapped mode corresponds to an (embedded) eigenvalue of the elasticity operator acting on functions which satisfy traction-free boundary conditions.  We assume that the elastic material is homogeneous and isotropic having zero Poisson's coefficient. We  generalise previous results obtained in \cite{HSW},  
where  the existence of at least two embedded eigenvalues  for the infinite strip with a horizontal crack has been proved. Moreover, for an elastic plate with crack it was shown that there exist  infinitely many eigenvalues and a one-sided asymptotic estimate for the distance of the eigenvalues to some spectral threshold has been given. We generalise this estimate and prove an asymptotic formula. Instead of the variational ansatz, which has been used in \cite{HSW}, we use a boundary integral method based on the  corresponding Dirichlet-to-Neumann mapping. 

As in \cite{HSW} we want to take advantage of the symmetry of the domain, which allows us to consider an equivalent mixed problem. An additional symmetry induced by the choice of Poisson's ratio allows us to prove the existence of exactly two discrete eigenvalues of  some symmetric part of the operator. For non-zero Poisson's ratio this last symmetry decomposition fails and the eigenvalues will in general  turn into resonances. Although, the  method is suitable to treat this problem,  this case  will not be considered in the present work. 
The symmetry decomposition of the elasticity operator for vanishing Poisson's ratio goes back to \cite{Roitetal}, where the existence of embedded eigenvalues was  shown for the elastic semistrip. These results have been generalised to the case of an elastic strip or plate with zero Poisson's ratio where material 
properties are perturbed \cite{FoeWeidl,Foe}, where  a hole is cut out \cite{Foe3D}, or strips and plates having a crack \cite{HSW}. In \cite{FoeWeidl,Foe} an  asymptotic formula for the convergence of the  eigenvalues to some spectral threshold was proved. The eigenvalue expansion is constructed via a Birman-Schwinger analysis of the original operator.  
In contrast to the perturbation of material properties, in the case of a shrinking crack the domain of the elasticity operator\footnote{or more precisely the domain of its quadratic form} does not remain unchanged. Hence, the Birman-Schwinger cannot be applied. Moreover,  the variational approach which was used in \cite{HSW} is not suitable to obtain
an asymptotic formula for the eigenvalues. In the present article we want to use a  boundary integral approach for the proof of the  asymptotic formula. This method allows us to reformulate
the singular perturbation of the original operator into an additive perturbation of the Dirichlet-to-Neumann operator. Then  the asymptotic formula for
the eigenvalues follows by a Birman-Schwinger analysis of the corresponding operator acting on the boundary.  The method was outlined in the simple case of the 
Dirichlet Laplacian with Neumann window in \cite{HaenelWeidl}.  

For bounded domains a similar approach has been employed in \cite{AmKaLee} to prove an asymptotic formula 
for the eigenvalues of  the Laplacian acting on a domain with a small inclusion, cf.\ also  \cite{AmKaLeeElastic13} for the treatment of an  elastic crack problem.

In addition in the case of an infinite elastic waveguide, there are various numerical approaches, see e.g.\ \cite{GridinCraster05,GridinAdaCraster05,ZernovPichKap06,ZernovKap08,LawrieKaplunov12,Pagneux12}. However, to our knowledge the articles  \cite{HSW,Roitetal,FoeWeidl,Foe,Foe3D} contain the only analytical results for a perturbed elastic strip or plate. 

\section{Statement of the problem}
Let  $I  := \left(-\pi/2, \pi/2\right)$. We denote by  $\Omega^{(d)} : = \RR^{d-1}  \times I$ an elastic medium in
dimension $d \in \{ 2,3\}$  and consider cracks of the form  
$\overline{\Gamma} \times \{0 \}$,  where $ \Gamma \subseteq \RR^{d-1}$ is a bounded open set. As there is no risk of confusion
we identify  $\overline{\Gamma} \times \{0 \}$ with $\overline{\Gamma}$. 
In what follows we denote by $u : \Omega^{(d)} \backslash \overline{\Gamma}  \to \CC^d$ 
the displacement field of the elastic material. Its strain is given by the matrix
\begin{equation}\label{def:strain}
	 E(u) := \frac{1}{2} ( \partial_i u_j + \partial_j u_i)_{i,j=1, \ldots, d}
\end{equation}
and for the stress we have 
\begin{equation}\label{def:stress}
	\Sigma(u) := 2 \mu  E(u) + \lambda  \mathrm{div} (u)  I_d ,
\end{equation}
where $I_d$ is the $d$-dimensional unit matrix   and $\lambda$ and $\mu$ are the  Lam\'e coefficients. 
These constants depend on  Poisson's ratio and  Young's modulus of the material and satisfy $\mu, \lambda + 2 \mu > 0$. In what follows we consider the case of a homogeneous material with zero 
Poisson's ratio, which implies for the Lam\'e coefficients  $\lambda =0$ and  $\mu=1$. Then the  elasticity operator with traction-free boundary conditions is defined by its  
energy form 
$$ a_\Gamma [u]: = \int_{\Omega \backslash \overline{\Gamma}}  \scal{E(u)}{\Sigma(u)}_{d \times d} \; \mathrm{d} x $$
where $u \in D[a_\Gamma] := H^1(\Omega\backslash \overline{\Gamma}; \CC^d)$. Here $\scal{\cdot}{\cdot }_{d \times d}$ denotes the standard scalar product in $\CC^{d\times d}$
identified with $\CC^{d^2}$. 
Then $a_\Gamma$ defines a positive quadratic form, which is closed by Korn's inequality, cf.\ e.g.\ \cite{HSW} and the references therein.  The corresponding 
self-adjoint operator $A_{\Gamma}$ acts as the matrix differential operator
\begin{equation}
  -  \Delta - \; \mathrm{grad} \ \mathrm{div} 
\end{equation}
on functions with vanishing conormal derivative, i.e., on $\partial ( \Omega^{(d)} \backslash \overline{\Gamma})$ we have
\begin{equation} \label{eq:conormal_derivative}
	\Sigma (u) \cdot {\bf n} = 0 
\end{equation}
for all $u :\Omega^{(d)}\backslash \overline{\Gamma} \to \CC^d$ in $D(A_\Gamma)$. Here $\bf n$ denotes the unit normal.
Due to the location of the crack and the assumptions on Poisson's ratio  we will use the symmetries induced by the domain  and the operator.
They allow us  to decompose the form domain and the operator domain into symmetric pieces.

\subsection{The symmetry decomposition in 2D} 
Using the ideas from \cite{HSW,Roitetal,FoeWeidl} 
we denote $\mathcal H:=L^2(\Omega^{(2)};\CC^2)$ and define the following subspace of symmetric waves
\begin{align}\label{def:H_s}
	 \mathcal H_{\mathrm{s}} :=\left\{u\in \mathcal H : u_1(x_1,x_2)= u_1(x_1,-x_2);\; u_2(x_1,x_2) = - u_2(x_1,-x_2) \right\}.
\end{align}
The elements of its orthogonal complement $\mathcal H_{\mathrm{as}} := \mathcal H_{\mathrm s}^\perp$ are called antisymmetric waves.
Furthermore, we introduce  
\begin{align}\label{def:H_1}
  \mathcal H_1:=  \biggr\{u \in \mathcal H : \partial_2 u_1(x_1,x_2)=0,\, u_2(x_1,x_2)=0 \biggr\} , 
\end{align}
and denote by $\mathcal H_2 = \mathcal H_{\mathrm{s}} \ominus \mathcal H_1$ its  orthogonal complement in $\mathcal H_{\mathrm{s}}$. Then 
\begin{align}\label{def:H_2}
\mathcal H_2 =\left\{ u \in \mathcal H_{\mathrm{s}} : \int_{-\frac{\pi}{2}}^{\frac{\pi}{2}} u_1(x_1,x_2) \; \mathrm{d} x_2 =0\ \text{for a.e. }
  x_1\in\RR \right\}
\end{align}
and we obtain  the following decomposition of Hilbert space
\begin{align*}
\mathcal H = \mathcal H_{\mathrm{s}} \oplus \mathcal H_{\mathrm{as}} = \mathcal H_1 \oplus \mathcal H_2 \oplus \mathcal H_{\mathrm{as}}.
\end{align*}
The form $a_\Gamma$  as well as the operator $A_\Gamma$ decompose in the same way
\begin{align*}
	  a_\Gamma &= a_{\Gamma}^{(\mathrm{s})} \oplus a_{\Gamma}^{(\mathrm{as})} = a_{\Gamma}^{(1)} \oplus a_{\Gamma}^{(2)} \oplus a_{\Gamma}^{(\mathrm{as})},\\[4pt]
  A_\Gamma &= A_{\Gamma}^{(\mathrm{s})} \oplus A_{\Gamma}^{(\mathrm{as})} = A_{\Gamma}^{(1)} \oplus A_{\Gamma}^{(2)} \oplus A_{\Gamma}^{(\mathrm{as})}, 
\end{align*}
where $a_{\Gamma}^{(\dagger)}$ and $A_{\Gamma}^{(\dagger)}$ act in $\mathcal H_\dagger$ for $\dagger \in \{\mathrm{s},\mathrm{as},1,2\}$, cf.\ \cite{HSW}. Based on the ideas
in \cite{Birman62} one proves that the essential spectrum of  $A_{\Gamma}^{(\dagger)}$ 
is  independent of the crack. We have 
\begin{equation} \sigma_{\mathrm{ess}}(A_\Gamma) = \sigma_{\mathrm{ess}}(A_\Gamma^{(\mathrm{s})}) = \sigma_{\mathrm{ess}}(A_\Gamma^{(\mathrm{as})}) = \sigma_{\mathrm{ess}}(A_\Gamma^{(1)}) = [ 0, \infty) \end{equation}
and 
\begin{equation}  \sigma_{\mathrm{ess}}(A_\Gamma^{(2)}) = [ \Lambda , \infty) \end{equation}
for some $\Lambda > 0$, cf.\ \cite{HSW}.  The role of $\Lambda$ is elucidated  in Section \ref{subsec:dispersion_curves}. 

\subsection{The symmetry decomposition in 3D} 
In the case of a three-dimensional plate we put $\mathcal H :=  L^2(\Omega^{(3)};\CC^3)$ as well as 
\begin{align}
	\mathcal H_{\mathrm{s}} &:= \{u\in  \mathcal H : u_k(x_1,x_2,x_3) = u_k(x_1,x_2,- x_3),\ k=1,2 \; \wedge \\
		& \phantom{:= \{} \qquad \qquad  u_3(x_1,x_2,x_3)= - u_3(x_1,x_2,- x_3) \} ,  	
\end{align}
and $\mathcal H_{\mathrm{as}} := (\mathcal H_{\mathrm s})^\perp$. We set 
\begin{align} 
	\mathcal H_1  &:=\left\{u \in \mathcal H : \partial_3 u_1 = \partial_3 u_2 =0,\;   u_3 = 0  \right\} 
\end{align}
and $\mathcal H_2 := \mathcal H_{\mathrm{s}} \ominus \mathcal H_1$. 
As above we obtain the decomposition $A_\Gamma =  A_{\Gamma}^{(1)} \oplus A_{\Gamma}^{(2)} \oplus A_{\Gamma}^{(\mathrm{as})}$, 
where $A_{\Gamma}^{(\dagger)}$ acts in $\mathcal H_\dagger$ for $\dagger \in \{\mathrm{s},\mathrm{as},1,2\}$ and 
\begin{equation} \sigma_{\mathrm{ess}}(A_\Gamma) = \sigma_{\mathrm{ess}}(A_\Gamma^{(\mathrm{s})}) = \sigma_{\mathrm{ess}}(A_\Gamma^{(\mathrm{as})}) = \sigma_{\mathrm{ess}}(A_\Gamma^{(1)}) = [ 0, \infty)  \end{equation}
and 
\begin{equation}  \sigma_{\mathrm{ess}}(A_\Gamma^{(2)}) = [ \Lambda , \infty) ,\end{equation}
cf.\ \cite{HSW}. The constant  $\Lambda$ is given as in the two-dimensional case.

\section{The results}
We generalise the one-sided asymptotic estimate  in \cite{HSW} and prove an asymptotic formula for the eigenvalues of the operator $A_{\Gamma}^{(2)} $ as the cracks size tends to zero. Moreover, we show an estimate on the  number of eigenvalues for small crack sizes. Since, the operator $A_{\Gamma}^{(2)} $ has at least two eigenvalues in the two-dimensional case and infinitely many eigenvalues in three dimensions we will need an additional  symmetry conditions on the shape of the crack. In doing so, we may prove an asymptotic formula for each eigenvalue. 

\subsection{The two-dimensional case} We assume that  $\Gamma \subseteq \RR$ is a  finite union of bounded open intervals and that $\Gamma$ is symmetric with respect to the axis $x_1 =0$, i.e., we have  $\Gamma = - \Gamma$. We put $\Gamma_\ell := \ell \cdot \Gamma$.
\begin{theorem}\label{th:main_2D}
	There exists $\ell_0 = \ell_0(\Gamma) > 0$ such that for $\ell \in (0,\ell_0)$ the operator 
	$A_{\Gamma_\ell}^{(2)}$ has exactly two discrete 
	eigenvalues $\lambda_1(\ell)$ and $\lambda_2(\ell) $ below its essential spectrum $[\Lambda, \infty)$. They satisfy
	\begin{align}
		\Lambda - \lambda_1(\ell) &= \nu_1 \cdot \ell^4  
		+ \mathcal O(\ell^5) \qquad \text{as} \quad \ell \to 0   , \\
		\Lambda - \lambda_2 (\ell)  &= \nu_2 \cdot \ell^8 
		 + \mathcal O(\ell^9) \qquad \text{as} \quad \ell \to 0  . 
	\end{align}
	The constants $\nu_1 = \nu_1(\Gamma) > 0$ and $\nu_2  = \nu_2(\Gamma) > 0$ are given by \eqref{def:nu_1} and \eqref{def:nu_2}. %
\end{theorem}

\subsection{The three-dimensional case}
We assume that $\Gamma = B(0,1) $ is the ball of radius  $1$ with centre $0$ and let   $\Gamma_\ell = \ell \cdot \Gamma = B(0,\ell)$ be the ball of radius $\ell$ and centre $0$. 
\begin{theorem}\label{th:main_3D}
	For $\ell > 0$ the  elasticity operator $A_{\Gamma_\ell}^{(2)}$ has infinitely many eigenvalues 
	below its essential spectrum $[\Lambda, \infty)$. There exists an enumeration of the eigenvalues  $\lambda_m(\ell)$, $m \in  \ZZ$, such that for each $m \in \ZZ$ we have the   following asymptotic expansion
	\begin{equation}
		\Lambda -  \lambda_m(\ell) =  \rho_m  \cdot \ell^{6+4|m|}    + \mathcal O(\ell^{7+4|m|} ) \qquad \text{as} \quad \ell \to 0 . 
	\end{equation} 
	The constants   $ \rho_m  > 0$ are given by \eqref{def:nu_m}. 
\end{theorem}

\section{The analysis of the unperturbed operator} \label{subsec:dispersion_curves}
As a first step we investigate  the unperturbed problem corresponding to $\Gamma = \varnothing$. 
\subsection{The two-dimensional case}
Put  $\Omega = \RR \times I$, where $ I = \left( - \pi/2 , \pi/2\right)$.  From general regularity theory for elliptic boundary value problems 
we know that the domain of $A_\varnothing$ is contained in $H^2(\Omega;\CC^2)$, i.e., we have 
\begin{equation*}
	D(A_\varnothing) = \{ u \in H^2(\Omega;\CC^2) :  \partial_1 u_2 + \partial_2 u_1 = 0 \; \wedge \; \partial_2 u_2 = 0 \text{ on } \partial \Omega\} .  
\end{equation*}
Applying the Fourier transform in the horizontal direction leads to a family of unbounded operators depending on a complex parameter
$\xi \in \CC$
\begin{equation} A_\varnothing(\xi) := \left( \begin{array}{cc} 2\xi^2 - \partial_2^2 & - \mathrm{i} \xi \partial_2 \\
- \mathrm{i} \xi \partial_2 & \xi^2 - 2 \partial_2^2  \end{array} \right)  \end{equation}
acting in $L_2(I ; \CC^2)$ having the operator domain 
\begin{equation} D(A_\varnothing(\xi)) := \left\{ u \in H^2 (I; \CC^2) :  
  \mathrm{i} \xi u_2 + \partial_2 u_1 = 0 \;  \wedge \; \partial_2 u_2 = 0 \text{ on } \{\pm \pi/2 \} \right\} .\end{equation}
The corresponding sesquilinear  form is given by 
\begin{align*}
	a_\varnothing(\xi)[u,v] := \int_{-\frac\pi2}^{\frac\pi2}  
 		& \xi^2( 2  u_1 \overline{v_1} +  u_2 \overline{v_2} )
 		 + \partial_2 u_1 \cdot \overline{\partial_2 v_1} + 2 \partial_2 u_2 \cdot \overline{\partial_2 v_2}  \\
		& - \mathrm{i} \xi (\partial_2 u_1 \cdot\overline{v_2}  -  u_2 \cdot \overline{\partial_2 v_1} )  \; \mathrm{d}  x_2
\end{align*}
with $u,v \in D[a_\varnothing(\xi)] := H^1(I; \CC^2)$. This definition extends to $\xi \in \CC$ and defines a holomorphic family of type (a) in the sense of Kato, cf.\ \cite[Chapter VII. \S 4, Paragraph 2]{Kato}. 
The decomposition of $\mathcal H$ into the spaces  $\mathcal H_\mathrm{s}$, $\mathcal H_{\mathrm{as}}$, $\mathcal H_1$ and $\mathcal H_2$
induces  the following decomposition in the Fourier image
\begin{equation*}
	L_2(I;\CC^2) = h_{\mathrm{s}} \oplus h_{\mathrm{as}} = h_1 \oplus h_2 \oplus h_{\mathrm{as}} ,
\end{equation*}
where 
\begin{align}
	h_{\mathrm{s}} &:= \{ u \in L_2(I; \CC^2) : u_1 (x_2) = u_1(-x_2) \;  \wedge \; u_2(x_2) = - u_2(-x_2) \} ,\\
 	h_{\mathrm{as}} &:= \{ u \in L_2(I; \CC^2) : u_1 (x_2) = -u_1(-x_2) \;  \wedge \;  u_2(x_2) =  u_2(-x_2) \} ,
\end{align}
as well as
\begin{align} \label{def:h_i}
	h_1 := \left\{ c \cdot \left(\begin{array}{c} 1\\ 0  \end{array} \right): c\in \CC \right\}, \qquad \qquad 
	h_2 := \left\{ u \in h_{\mathrm{s}}: \int_{-\frac\pi2}^{\frac\pi2} u_1(x_2) \; \mathrm{d} x_2 = 0 \right\}.
\end{align}
The following lemma shows that the operator $A_\varnothing(\xi)$ decomposes for every $\xi\in \CC$ into
an orthogonal sum of self-adjoint operators acting in $h_{\mathrm{s}}$ and  $h_{\mathrm{as}}$ respectively in
 $h_1$, $h_2$ and $h_{\mathrm{as}}$. 
\begin{lemma}[cf.\ \cite{HSW}]\label{lemma:direct_integral}
The following two assertions hold true:
\begin{enumerate}
  	\item For every $\xi \in \CC$ the operator $A_\varnothing(\xi)$ decomposes into 
  	$$A_\varnothing(\xi) = A_\varnothing^{(\mathrm{s})}(\xi) \oplus A_\varnothing^{(\mathrm{as})}(\xi) = A_\varnothing^{(1)}(\xi) \oplus A_\varnothing^{(2)}(\xi) \oplus A_\varnothing^{(\mathrm{as})}(\xi) ,$$
  	where $A_\varnothing^{(\dagger)}(\xi)$ acts in $h_\dagger $ for $\dagger  \in \{\mathrm{s}, \mathrm{as},1,2\}$. 
  	\item The  operator $A_\varnothing$  is unitarily equivalent to the direct integral of the $A_\varnothing(\xi)$'s, i.e., we have 
  	$$ A_\varnothing \cong \int^\oplus_\RR A_\varnothing(\xi) \; \mathrm{d} \xi .$$
  	Moreover, for $\dagger \in \{\mathrm{s}, \mathrm{as},1,2\}$ we have 
  	$$ A_\varnothing^{(\dagger)} \cong \int^\oplus_\RR A_\varnothing^{(\dagger)}(\xi) \; \mathrm{d} \xi . $$
\end{enumerate}
\end{lemma}
For every $\xi \in \CC$ the spectrum $\sigma( A_\varnothing(\xi))$ consists of a discrete set 
of eigenvalues of finite algebraic multiplicity as $D(A_\varnothing(\xi))$ is compactly embedded into $L_2(I; \CC^2)$, and thus, the 
resolvent is compact, cf.\ \cite[Theorem III.6.29]{Kato}. Since  $(a_\varnothing (\xi))_{\xi \in \CC}$ forms an analytic family of 
type (a) the eigenvalues depend holomorphically  on $\xi$ (with the possible exception
of algebraic branching points) and the  direct integral representation implies 
\begin{align}\label{eq:spectrum_direc_integral}
	\sigma(A_\varnothing) = \bigcup_{\xi \in \RR} \sigma (A_\varnothing(\xi))  \quad \text{and} \quad 
	\sigma(A_\varnothing^{(\dagger)}) = \bigcup_{\xi \in \RR} \sigma (A_\varnothing^{(\dagger)}) 
\end{align}
for  $\dagger \in \{ \mathrm{s},\mathrm{as}, 1,2\}$. 

Now we consider the eigenvalue distribution of the operator $A_\varnothing(\xi)$ depending on the parameter $\xi \in \RR$. The arising curves  are generally referred to as the dispersion curves of $A_\varnothing$. In what follows the  dispersion curve  corresponding to the lowest eigenvalue of $A_\varnothing^{(2)}$ is of particular interest since it describes  the behaviour of the unperturbed operator near the  spectral threshold $\Lambda$.  
More generally we fix $\xi \in \CC$ and choose $\omega \in \CC$, $u \in H^2(I; \CC^2 )$ such that 
\begin{align}\label{eq:dispersion_curves}
  	A(\xi) u - \omega u &= 0 \qquad   \text{in } I, \\
  	\label{eq:dispersion_curves1}
  	\mathrm{i} \xi u_2 + \partial_2 u_1 &= 0 \qquad \text{on } \{\pm \pi/2 \},\\
  	\label{eq:dispersion_curves2}
  	\partial_2 u_2 &= 0 \qquad \text{on } \{\pm \pi/2 \}.
\end{align}
Here and subsequently we denote by $A(\xi)$ the matrix differential operator
$$ A(\xi) :=  \left( \begin{array}{cc} 2\xi^2 - \partial_2^2 & - \mathrm{i} \xi \partial_2 \\
- \mathrm{i} \xi \partial_2 & \xi^2 - 2 \partial_2^2  \end{array} \right) $$
if no boundary conditions are specified. The assumptions on $u$ already imply that $u \in C^\infty \left( [-\frac\pi2,\frac\pi2];\CC^2 \right) $. In 
order to solve \eqref{eq:dispersion_curves}-\eqref{eq:dispersion_curves2} 
we have to distinguish several cases:
\begin{enumerate}
\item Let $\omega \in \CC \backslash\{0\}$, $\omega \notin \{ \xi^2, 2 \xi^2\}$: 
A fundamental solution of the differential equation  \eqref{eq:dispersion_curves}  is given by the functions
\begin{align*}
  	v^1(x_2) &:= \left(\begin{array}{c}  \beta \\ -  \xi \end{array} \right) e^{\mathrm{i} \beta x_2};
  	&v^3(x_2) &:= \left(\begin{array}{c}  \xi \\  \gamma  \end{array} \right) e^{\mathrm{i} \gamma x_2};\\[8pt]
  	v^2(x_2) &:= \left(\begin{array}{c} - \beta \\ - \xi \end{array} \right) e^{- \mathrm{i} \beta x_2};  
  	& v^4(x_2) &:= \left(\begin{array}{c}  \xi \\ - \gamma \end{array} \right) e^{- \mathrm{i}  \gamma x_2}, 
\end{align*}
where  
$$ \beta := \sqrt{ \omega  - \xi^2} , \qquad \gamma := \sqrt{ \frac{\omega}2 - \xi^2}.$$
Here we choose the branch of the square root function such that $z \mapsto \sqrt{z}$ is holomorphic for 
$z \in \CC \backslash (- \infty,0]$ and  $\sqrt{z} > 0$ for all $z > 0 $. Moreover for $z \le 0  $ we define $\sqrt{z}$
such that $\Im (\sqrt{z}) \ge 0$. 
\item Let $\omega \in \CC \backslash \{0\}$, $\omega = \xi^2$: In this case we have $\beta = 0$ and the fundamental solution of  \eqref{eq:dispersion_curves} reads as
\begin{align*}
  v^1(x_2) &:= \left(\begin{array}{c} 0 \\ 1 \end{array} \right);
  &v^3(x_2) &:= \left(\begin{array}{c}  \xi \\  \gamma  \end{array} \right) e^{\mathrm{i} \gamma x_2}; \\[8pt]
  v^2(x_2) &:= \left(\begin{array}{c} 1\\  - \mathrm{i}   \xi x_2 \end{array} \right);
  & v^4(x_2) &:= \left(\begin{array}{c}  \xi \\ - \gamma \end{array} \right) e^{- \mathrm{i} \gamma x_2}.
\end{align*}
\item Let $\omega \in \CC \backslash\{0\}$, $\omega = 2 \xi^2 $: Then  $\gamma = 0$ and a fundamental solution of  \eqref{eq:dispersion_curves} is given by 
\begin{align*}
  v^1(x_2) &:= \left(\begin{array}{c}  \beta \\ -  \xi \end{array} \right) e^{\mathrm{i} \beta x_2};
  &v^3(x_2) &:= \left(\begin{array}{c} 1\\ 0\end{array} \right);\\[8pt]
  v^2(x_2) &:= \left(\begin{array}{c} -\beta \\ - \xi \end{array} \right) e^{- \mathrm{i} \beta x_2};
  & v^4(x_2) &:= \left(\begin{array}{c} \mathrm{i} \xi x_2\\ 1\end{array} \right).
\end{align*}
\item Let $\omega = 0, \xi = 0$: In this case all solutions have to be linear in both components.
\item Let $\omega = 0, \xi \neq 0$: Then we have $\beta = \gamma = \mathrm{i} \xi$ and we get 
\begin{align*}
 	 v^1(x_2) &:= \left(\begin{array}{c} \xi \\ \mathrm{i} \xi  \end{array} \right) e^{-\xi x_2};
  	& v^3(x_2) &:= \left[x_2 \left(\begin{array}{c} \xi \\  \mathrm{i} \xi  \end{array} \right) +  
  	\left( \begin{array}{c}  0\\[4pt]   3 \mathrm{i} \end{array}\right)\right] e^{-\xi x_2};\\[8pt]
 	 v^2(x_2) &:= \left(\begin{array}{c} \xi  \\ - \mathrm{i} \xi \end{array} \right) e^{\xi x_2}; 
  	&v^4(x_2) &:= \left[x_2 \left( \begin{array}{c} \xi  \\ - \mathrm{i} \xi \end{array} \right) + 
  	 \left( \begin{array}{c}  0 \\[4pt]  3 \mathrm{i} \end{array}\right) \right] e^{ \xi x_2}.
\end{align*}
\end{enumerate}
Considering  the boundary conditions leads to the well-known Rayleigh-Lamb equations for the symmetric and  antisymmetric part of 
the operator. The following lemma deals only with the symmetric part. 
\begin{lemma}\label{lemma:spectrum}
Let $\xi \in \CC$. Then the following assertions hold true:
\begin{enumerate}
	\item We have $\sigma(A_\varnothing^{(1)}(\xi)) = \{ 2 \xi^2\} $. 
  	\item The eigenvalues of $A_\varnothing^{(2)}(\xi)$ are simple. 
 	 \item Let $\omega \in \CC \backslash\{0\}$. Then $\omega \in \sigma(A_\varnothing^{(2)}(\xi))$ if and only if 
  	\begin{align} \label{eq:Rayleigh-Lamb}
   	   \frac{\sin \left(\beta \frac{\pi}2\right)}{\beta}  \cos \left(\gamma \frac{\pi}2\right)
   	   \gamma^2 + \cos \left(\beta \frac{\pi}2\right) \frac{\sin \left(\gamma \frac{\pi}2\right)}{\gamma} \xi^2 = 0, 
  	\end{align}
  	where as before
  	$$ \beta = \sqrt{ \omega  - \xi^2} , \qquad \gamma = \sqrt{ \frac{\omega}2 - \xi^2}.$$ 
  	Note that $0 \notin  \sigma(A_\varnothing^{(2)}(\xi))$ for any $\xi \in \CC$.
  	If $\gamma \neq 0$, then a  (non-normalised)  eigenfunction is  given by  
  	\begin{equation}\label{eq:eigenfuncion_elast}
  		x_2 \mapsto  \begin{pmatrix} \mathrm{i} \gamma^2 \beta \cos\left(\frac\pi2 \gamma\right) \cos\left(\beta x_2 \right) 
  		+ \mathrm{i} \xi^2 \beta \cos\left(\frac\pi2 \beta \right) \cos\left(\gamma x_2 \right) \\[6pt]
  		\xi \gamma^2  \cos  \left(\frac\pi2 \gamma\right) \sin \left(\beta x_2 \right)  - \xi \beta \gamma
  		\cos \left(\frac\pi2 \beta \right) \sin \left(\gamma x_2 \right)  \end{pmatrix} .
  	\end{equation} 	
\end{enumerate}
\end{lemma}
The assertion of  Lemma \ref{lemma:spectrum} is well known. 
In what follows for  real $\xi \in \RR$ we denote by $\zeta_1(\xi) < \zeta_2(\xi) < \ldots $ the (simple) eigenvalues of the operator 
$A_\varnothing^{(2)} (\xi)$. Obviously the functions $\zeta_k(\cdot)$ are   real analytic. 
Let  
\begin{equation}
	\Lambda := \inf \{ \zeta_1  (\xi): \xi \in \RR \} . 
\end{equation}
F\"orster and Weidl showed in \cite{FoeWeidl} that $\Lambda=1.887837\pm10^{-6}  > 0$ and that the infimum
is achieved for  $\xi = \pm \varkappa$, where 
\begin{equation}
	\varkappa =  0.632138\pm 10^{-6} >0 . 
\end{equation}
In particular, from \eqref{eq:spectrum_direc_integral} and the invariance of the essential spectrum we obtain 
\begin{equation}\label{eq:ess_spec_A_02}
	\sigma(A_\varnothing^{(2)} ) = \sigma(A_{\Gamma}^{(2)}) = [\Lambda  , \infty ) \subsetneq [ 0, \infty) 
\end{equation}
for any crack $\Gamma \subseteq \RR$, which is given by a finite union of bounded intervals. 

\subsection{The three-dimensional case}
In this case the unperturbed elasticity operator acts as   the matrix differential operator 
\begin{equation*}
	- \begin{pmatrix} 2 \partial_1^2 + \partial_2^2 + \partial_3^2 & \partial_1 \partial_2 & \partial_1 \partial_3 \\[4pt] 
		\partial_1 \partial_2 & \partial_1^2 + 2 \partial_2^2 + \partial_3^2 & \partial_2 \partial_3 \\[4pt]
		\partial_1 \partial_3 & \partial_2 \partial_3 & \partial_1^2 + \partial_2^2 + 2 \partial_3^2 \end{pmatrix} 
\end{equation*}
with the operator  domain 
\begin{align*}
	D(A_\varnothing ) = \left\{ u \in H^2(\Omega;\CC^3 ) : \partial_3 u_k  + \partial_k u_3  = 0 \text{ on }  \partial \Omega 
	\text{ for } k= 1, 2,3  \right\}  . 
\end{align*}
Applying the Fourier transform with respect to the first two variables we obtain
a  family of sectorial operators $(A_\varnothing (\xi))_{\xi \in \RR^2}$ acting on $L^2(I; \CC^3)$,   
\begin{equation}
	A_\varnothing (\xi) : = \begin{pmatrix}  2 \xi_1^2 + \xi_2^2 - \partial_3^2 & \xi_1 \xi_2 & - \mathrm{i} \xi_1 \partial_3 \\[4pt]
	 \xi_1 \xi_2 	&  \xi_1^2 + 2 \xi_2^2 - \partial_3^2	& - \mathrm{i} \xi_2 \partial_3 \\[4pt]
	- \mathrm{i} \xi_1 \partial_3 & - \mathrm{i} \xi_2 \partial_3 & \xi_1^2  + \xi_2^2 - 2 \partial_3^2 . 
	\end{pmatrix} . \end{equation}
The domain of the operator $D(A_\varnothing (\xi))$ consists of those functions $u \in H^2( I; \CC^3)$, which satisfy 
\begin{equation}
	\begin{pmatrix} \partial_3 u_1  + \mathrm{i} \xi_1 u_3 \\
	\partial_3 u_2 +  \mathrm{i} \xi_2 u_3 \\
	\partial_3 u_3 \end{pmatrix} (x_3)  = \begin{pmatrix} 0 \\ 0 \\ 0 \end{pmatrix} \quad  \text{for } x_3 = \pm \frac\pi2 . 
\end{equation}
As above we define 
\begin{align}
	h_{\mathrm{s}} &:= \{ u \in L_2(I; \CC^3) : u_k (x_3) = u_k(-x_3) \text{ for } k =1,2\; \wedge \;  u_3(x_3) = - u_3(-x_3) \} ,\\
  	h_{\mathrm{as}} &:= \{ u \in L_2(I; \CC^3) :  u_k (x_3) = - u_k(-x_3) \text{ for } k =1,2 \; \wedge \;  u_3(x_3) = u_3(-x_3) \} ,
\end{align}
as well as
\begin{align} 
  	h_1 := \left\{ c \cdot \left(\begin{array}{c} 1\\ 0 \\ 0 \end{array} \right): c\in \CC \right\}, \qquad \qquad 
 	 h_2 := \left\{ u \in h_{\mathrm{s}}: \int_{-\frac\pi2}^{\frac\pi2} u_k(x_3) \; \mathrm{d} x_3 = 0, \;  k=1,2 \right\}.
\end{align}
The operator $A_\varnothing (\xi)$ admits the following decomposition  
$$ A_\varnothing (\xi) = A_\varnothing ^{(\mathrm s)}(\xi) \oplus A_\varnothing ^{(\mathrm{as})}(\xi) =  A_\varnothing ^{(1)}(\xi) \oplus A_\varnothing ^{(2)}(\xi)  \oplus  A_\varnothing ^{(\mathrm{as})}(\xi), $$
where $A_\varnothing ^{(\dagger)}(\xi)$ acts in $h_{\dagger}$ for $\dagger \in \{ \mathrm{s}, \mathrm{as}, 1,2 \}$. We have
$$ A_\varnothing \cong \int^\oplus_\RR A_\varnothing(\xi) \; \mathrm{d} \xi \qquad \text{and} \qquad A_\varnothing^{(\dagger)} \cong \int^\oplus_\RR A_\varnothing^{(\dagger)}(\xi) \; \mathrm{d} \xi $$
for  $\dagger \in \{\mathrm{s}, \mathrm{as},1,2\}$. 
\begin{lemma}\label{lemma:spectral_elast_3D}
	For  $\xi \in \RR^2$ and $M \in \mathrm{SO}(2)$ the following assertions hold true:
	\begin{enumerate}
  		\item Let $u = (u_1,u_2,u_3) \in H^2(I; \CC^3)$. Then 
  		\begin{equation}\label{eq:rotational_A(xi)}
  			A_\varnothing (M\xi) \begin{pmatrix} M & 0 \\ 0 &1 \end{pmatrix} u  = \begin{pmatrix} M & 0 \\ 0 &1 \end{pmatrix} A_\varnothing (\xi) u . 
  		\end{equation}
		\item We have $\sigma(A_\varnothing^{(2)}(M \xi)) =  \sigma(A_\varnothing^{(2)}(\xi))$.  More precisely, 
		$\omega$ is an eigenvalue of the operator $A_\varnothing^{(2)}(\xi)$ if and only if 
		\begin{align} \label{eq:Rayleigh-Lamb_3D}
			\frac{\sin \left(\beta \frac{\pi}2\right)}{\beta}  \cos \left(\gamma \frac{\pi}2\right)
   	   		\gamma^2 + \cos \left(\beta \frac{\pi}2\right) \frac{\sin \left(\gamma \frac{\pi}2\right)}{\gamma} |\xi|^2 = 0, 
  		\end{align}
  	where 
  	$$ \beta = \sqrt{ \omega  - |\xi|^2} , \qquad \gamma = \sqrt{ \frac{\omega}2 - |\xi|^2}.$$ 
	\end{enumerate}
\end{lemma}
Lemma \ref{lemma:spectral_elast_3D} implies that the  eigenvalues of the operator $A_\varnothing^{(2)}(\xi)$ depend only on the
norm  of $\xi \in \RR^2$ and coincide with the eigenvalues  of the operator $A_\varnothing^{(2)}(|\xi|)$ arising in two dimensions.
For $\xi =( |\xi|  \cos \alpha, |\xi|  \sin \alpha)^T$, $\alpha \in [0, 2\pi)$, 
and $\gamma \neq 0$  the eigenfunctions  of the operator $A_\varnothing^{(2)}(\xi)$ are  given by 
\begin{equation}\label{eq:eigenfunction_elast_3D}
	x_3 \mapsto   \begin{pmatrix} M_\alpha & 0  \\ 0 & 1  \end{pmatrix}   
  		\begin{pmatrix} \mathrm{i} \gamma^2 \beta \cos\left(\frac\pi2 \gamma\right) \cos\left(\beta x_3 \right) 
  		+ \mathrm{i} |\xi|^2 \beta \cos\left(\frac\pi2 \beta \right) \cos\left(\gamma x_3 \right) \\[6pt] 0 \\[6pt] 
  		|\xi| \gamma^2  \cos  \left(\frac\pi2 \gamma\right) \sin \left(\beta x_3 \right)  - |\xi| \beta \gamma
  		\cos \left(\frac\pi2 \beta \right) \sin \left(\gamma x_3 \right)  \end{pmatrix} ,
\end{equation} 
where we put 
$$ M_\alpha  = \begin{pmatrix} \cos \alpha & - \sin \alpha  \\ \sin \alpha &\cos \alpha  \end{pmatrix} \in \mathrm{SO}(2). $$
Recall that 
\begin{equation*} 
  		x_2 \mapsto \begin{pmatrix} \mathrm{i} \gamma^2 \beta \cos\left(\frac\pi2 \gamma\right) \cos\left(\beta x_2\right) 
  		+ \mathrm{i} |\xi|^2 \beta \cos\left(\frac\pi2 \beta \right) \cos\left(\gamma x_2 \right) \\[6pt]  
  		|\xi| \gamma^2  \cos  \left(\frac\pi2 \gamma\right) \sin \left(\beta x_2 \right)  - |\xi| \beta \gamma
  		\cos \left(\frac\pi2 \beta \right) \sin \left(\gamma x_2\right)  \end{pmatrix} 
\end{equation*} 
is  an  eigenfunction of  the operator $A_\varnothing^{(2)} (|\xi|)$. In particular,  we obtain for the three-dimensional plate 
$$ \sigma_{\mathrm{ess}} (A_\varnothing^{(2)}) = [\Lambda , \infty) , $$
with the same constant $\Lambda > 0$ as for the infinite strip. 

\section{Proof of the main result - 2D}
\subsection{An equivalent  mixed problem}
The aim of this section is the construction of the Dirichlet-to-Neumann operator corresponding to
our problem. For this purpose we provide the boundary data on the crack and construct the solution of the inhomogenuous problem on the
upper part of the strip $\Omega_+ := \RR \times I_+$ with $I_+ :=  (0, \frac\pi2)$. This will be sufficient as 
we are only interested in the symmetric part of the operator. The latter is unitarly equivalent  to the operator $A_{\Gamma+}$, which is induced by the quadratic form
\begin{align}
	a_{\Gamma+} [u] &:= \int_{\Omega_+} \scal{E(u)}{\Sigma(u)}_{2 \times 2}   \; \mathrm{d}  x
\end{align}
with $ u \in D[ a_{\Gamma+}] := \{ u \in H^1(\Omega_+;\CC^2) : \mathrm{supp} ( u_2|_{\RR \times \{0\}}) \subseteq \overline{\Gamma} \}. $
Here $E(u)$ and $\Sigma(u)$ are given as in \eqref{def:strain} and \eqref{def:stress}.
Then $A_{\Gamma+}$ acts as the elasticity operator  $- \Delta - \; \mathrm{grad} \ \mathrm{div}$ and its domain contains exactly those functions 
$u \in H^1(\Omega_+;\CC^2)$ 
which satisfy 
\begin{align}
	&\left\{  \begin{array}{rll} \displaystyle \partial_1 u_2 + \partial_2 u_1 &= 0   &   \text{on } \RR \times \left\{\frac\pi2 \right\},  \\[4pt]
	2 \partial_2 u_2  &= 0  &   \text{on } \RR \times \left\{\frac\pi2\right\}, \end{array} \right. 
	\intertext{and}
	&\left\{  \begin{array}{rll} \displaystyle - ( \partial_1 u_2 +  \partial_2 u_1 ) & = 0 &  \text{on } \RR \times\{0\}, \\[4pt]
 	 -2 \partial_2 u_2 &= 0 &  \text{on } \Gamma ,\\[4pt]
 	 u_2 &= 0  &   \text{on } (\RR \times \{0\}) \backslash \overline{\Gamma }. \end{array} \right. 
 \end{align} 
As before these identities  should be understood in the weak sense. As a particular consequence  we have reduced the original problem to a mixed problem. 
The Hilbert space $\mathcal H_+ := L_2(\Omega_+; \CC^2)$ decomposes 
into an orthogonal sum
$$ \mathcal H_+ = \mathcal H_{1+} \oplus \mathcal H_{2+} $$
with 
\begin{equation}\label{eq:symmetry_omega_+}
  \mathcal H_{2+} := \left\{ u \in L_2(\Omega_+;\CC^2): \int_0^{\frac\pi2}  u_1(x_1,x_2 ) \; \mathrm{d}  x_2 = 0 
  \text{ for a.e. } x_1 \in \RR \right\}
\end{equation} 
and $\mathcal H_{1+} := (\mathcal H_{2+})^\perp$.
The form and the operator decompose in the same way in an orthogonal sum
$$  a_{\Gamma+} =  a_{\Gamma+}^{(1)} \oplus  a_{\Gamma+}^{(2)} , \qquad 
A_{\Gamma+} =  A_{\Gamma+}^{(1)} \oplus  A_{\Gamma+}^{(2)} ,$$ 
where $a_{\Gamma+}^{(\dagger)}$ and  $A_{\Gamma+}^{(\dagger)}$ acts in $\mathcal H_{\dagger+}$ for $\dagger \in \{1,2\}$. 
Next  we define the analogues of the spaces $h_i$, $i=1,2$ which were introduced in  \eqref{def:h_i}. We  denote by 
$h_{1+}$  the linear span of the constant function $(1,0)^T$ and let 
$h_{2+} = h_{1+}^\perp$, 
\begin{equation}
	h_{2+} = \left\{ u \in L_2(I_+;\CC^2) : \int_0^{\frac\pi2} u_1(x_2) \; \mathrm{d}  x_2 = 0 \right\} .
\end{equation}
The Fourier transform in the horizontal direction leads as before  to the parameter-dependent operator 
$$ A_{\varnothing+}(\xi) := \begin{pmatrix} 2 \xi^2 - \partial_2^2 & - \mathrm{i} \xi \partial_2 \\  - \mathrm{i} \xi \partial_2 &
	\xi^2 -  2 \partial_2^2 \end{pmatrix} $$
which acts in the Hilbert space $L_2(I_+;\CC^2)$ having  the operator domain 
$$ D(A_{\varnothing+}(\xi)) := \{ u \in H^2(I_+;\CC^2) : \mathrm{i} \xi u_2  + \partial_2 u_1= 0  \; \wedge \;   \partial_2 u_2 = 0 
	\text{ on } \{0 , \pi/2\} \}  .  $$
We note that the   subspaces $h_{1+}$ and $h_{2+}$ reduce the operator $A_{\varnothing+}(\xi)$, giving rise to self-adjoint operators
$A_{\varnothing+}^{(1)}(\xi)$ and $A_{\varnothing+}^{(2)}(\xi)$. Clearly, the operators $A_{\varnothing+}^{(i)}(\xi)$ and $A_{\varnothing}^{(i)}(\xi)$, $i =1,2$,  are unitarily equivalent. 

\subsection{The construction of the Dirichlet-to-Neumann operator}
In the first instance we consider  $\omega \in \CC \backslash\{0\}$ and let  $g \in H^{1/2}(\RR)$. 
We want to search for a solution $u \in H^1(\Omega_+ ;\CC^2)$ of the eigenvalue problem 
\begin{equation}\label{eq:solution_operator1}
  (- \Delta - \; \mathrm{grad} \ \mathrm{div}) u =   \omega u  , \qquad \text{in } \Omega_+  ,
\end{equation}
which satisfies the boundary data
\begin{align}\label{eq:solution_operator2}
  	\left\{ \begin{array}{rll} \partial_2 u_1 + \partial_1 u_2 &= 0  & \text{on } \RR \times \left\{\frac\pi2\right\} ,
 		\\[4pt] 2 \partial_2 u_2 & = 0  & \text{on } \RR \times \left\{\frac\pi2\right\} , \end{array} \right. \\ 
    \left\{ \begin{array}{rll} \partial_2 u_1 + \partial_1 u_2 &= 0 & \text{on } \RR \times \left\{0 \right\} ,
 		\\[2pt]  u_2 & = g & \text{on } \RR \times \left\{0 \right\} , \end{array} \right.  
 	\label{eq:solution_operator3}
\end{align}
Note that  conditions \eqref{eq:solution_operator2}-\eqref{eq:solution_operator3} should be understood in their 
variational form, i.e., the assertions \eqref{eq:solution_operator1}-\eqref{eq:solution_operator3} are equivalent 
to 
\begin{align}\label{eq:Poisson_variational}
	\int_{\Omega_+} \scal{E(u)}{\Sigma(v) }_{2\times 2}  \; \mathrm{d}  x = \omega \scal{u}{v}_{\Omega_+}  , 
\end{align}
which should  hold for all $v \in H^1(\Omega_+;\CC^2)$ with $v_2|_{\RR \times \{0\}} = 0$. 
Let $\hat u$ be the Fourier transform of $u$ taken in the horizontal direction. Then we have for $(\xi ,x_2) \in \RR \times I_+$
\begin{align}
	\label{eq:Poisson_FT1_elast}
	A(\xi) \hat u(\xi,x_2 ) &= \omega \hat u(\xi,x_2) , 
\end{align}
as well as  
\begin{align}
  	\left\{ \begin{array}{rl} \partial_2 \hat  u_1 \left(\xi, \frac\pi2 \right) + \mathrm{i} \xi \hat u_2  \left(\xi, \frac\pi2 \right) &= 0 ,
		\qquad \xi \in \RR \\[6pt] 2 \partial_2 \hat  u_2 \left(\xi, \frac\pi2 \right)& = 0  ,
		\qquad \xi \in \RR \hspace{-0.05cm} \phantom{\hat g(\xi)}  \end{array} \right.  &   \\[6pt] 
    \left\{ \begin{array}{rl} \partial_2 \hat  u_1 (\xi, 0) + \mathrm{i} \xi \hat  u_2  (\xi, 0) &= 0 , \qquad \xi \in \RR 
 		\\[5pt]  \hat  u_2  (\xi, 0) & = \hat g(\xi) , \quad \xi \in \RR  .  \end{array} \right.  & 
 		\label{eq:Poisson_FT3_elast}
\end{align}
Note that this already implies $\hat u(\xi, \cdot) \in C^{\infty} \left(\left[0, \frac\pi2 \right]; \CC^2 \right)$ for almost every 
$\xi \in \RR$. Since   $\omega \neq 0$ the solution $\hat u$ is for $\xi^2  \notin \{ \omega, \frac\omega2\}$ 
a linear combination of %the following functions
\begin{align*}
   v^1(x_2) &:= \left(\begin{array}{c}  \beta \\ -  \xi \end{array} \right) e^{\mathrm{i}\beta x_2};
  &v^3(x_2) &:= \left(\begin{array}{c}  \xi \\  \gamma  \end{array} \right) e^{\mathrm{i} \gamma x_2};\\[8pt]
  v^2(x_2) &:= \left(\begin{array}{c} - \beta \\ - \xi \end{array} \right) e^{- \mathrm{i}\beta x_2};  
  & v^4(x_2) &:= \left(\begin{array}{c}  \xi \\ - \gamma \end{array} \right) e^{- \mathrm{i} \gamma x_2} ,
\end{align*}
where  
$$ \beta= \sqrt{ \omega  - \xi^2}  \qquad \text{and} \qquad \gamma = \sqrt{ \frac{\omega}2 - \xi^2}.$$
Thus, we obtain for $(\xi, x_2)  \in \RR \times I_+$ 
$$  \hat u (\xi, x_2) = \sum_{k=1}^4 a_{k} (\xi,\omega) v^k(x_2)    $$
with  coefficients $a_1(\xi, \omega) , \ldots, a_4(\xi, \omega)$. Inserting the  boundary conditions  leads to a linear system 
$ L(\xi, \omega) a(\xi, \omega ) = b(\xi)$, where 
\begin{equation}
	L(\xi, \omega) := \left( \begin{array}{cccc}
  (\beta^2 - \xi^2) \ \mathrm{e}^{\mathrm{i} \beta \frac{\pi}2}  & (\beta^2 - \xi^2) \  \mathrm{e}^{-\mathrm{i} \beta \frac{\pi}2}  & 
	2\gamma\xi  \ \mathrm{e}^{\mathrm{i} \gamma \frac\pi2} & - 2\gamma \xi  \ \mathrm{e}^{-\mathrm{i} \gamma \frac\pi2}\\[4pt]
  - 2 \beta \xi \ \mathrm{e}^{\mathrm{i} \beta \frac{\pi}2} & 2 \beta \xi \ \mathrm{e}^{-\mathrm{i} \beta \frac{\pi}2} & 2 \gamma^2 \ \mathrm{e}^{\mathrm{i} \gamma \frac\pi2} & 
	2 \gamma^2 \ \mathrm{e}^{-\mathrm{i} \gamma \frac\pi2} \\[4pt]
   \mathrm{i} (\beta^2 - \xi^2) &  \mathrm{i} (\beta^2 - \xi^2)	&  2 \mathrm{i} \gamma \xi	& -2 \mathrm{i} \gamma \xi \\[4pt]
  - \xi  & - \xi   &  \gamma & - \gamma\\[4pt]
\end{array}\right)
\end{equation}
and
$$ a(\xi, \omega )  := \left( \begin{array}{c} a_1(\xi, \omega )  \\ a_2(\xi, \omega )  \\ a_3(\xi, \omega )  \\ a_4(\xi, \omega )  \end{array} \right), \qquad
b(\xi) := \left(\begin{array}{c} 0\\ 0 \\ 0 \\ \hat g(\xi) \end{array} \right) . $$
Thus, $a(\xi, \omega ) = L (\xi, \omega )^{-1} b(\xi)$ provided $L(\xi)$ is invertible. We note that the determinant of $L(\xi)$ is given by 
$$ \det( L(\xi, \omega )) = 32  \gamma^2 (\gamma^2 + \xi^2)  \  \Bigr[ \sin \left(\beta \frac{\pi}2\right) \cos \left(\gamma \frac{\pi}2\right)
  \gamma^3 + \cos \left(\beta \frac{\pi}2\right) \sin \left(\gamma \frac{\pi}2\right) \beta \xi^2 \Bigr] , 
$$
which coincides up to a factor with  the left-hand  side of the  Rayleigh-Lamb equation \eqref{eq:Rayleigh-Lamb}. 
In particular if $\omega \in \CC \backslash [0,\infty)$  then 
$L(\xi, \omega )$ is invertible for every $\xi \in \RR$.
Put  
\begin{align*}
  	h (x)  &= - 2 \partial_2 u_2 (x ,0) , \qquad x \in \RR . 
\end{align*}
Then its Fourier transform $\hat h$ satisfies $\hat h(\xi) = R(\xi,\omega) a(\xi,\omega),$ where 
$$ R(\xi,\omega) = \left( \begin{array}{cccc}  
	2 \mathrm{i} \beta \xi		& - 2\mathrm{i} \beta \xi		& -2 \mathrm{i} \gamma^2 	& - 2\mathrm{i} \gamma^2 
	\end{array} \right) .
$$
An elementary calculation shows that 
$$ \hat h(\xi) = m_\omega(\xi) \hat g(\xi)$$
where 
  \begin{equation}
	  m_\omega(\xi) :=  \frac{-2 \sin\left(\frac{\beta\pi}2\right) \sin \left( \frac{\gamma\pi}2\right)
  \left[ \gamma^6 + 2 \gamma^2 \xi^4 + \xi^6\right] + 4 \left[ \cos\left( \frac{\beta\pi}2\right) \cos\left( \frac{\gamma\pi}2\right) - 1\right]
  \beta \gamma^3 \xi^2}{(\gamma^2 + \xi^2) [ \sin \left( \frac{\beta\pi}2\right) \cos \left( \frac{\gamma\pi}2\right)
  \gamma^3 + \cos \left( \frac{\beta\pi}2\right) \sin \left( \frac{\gamma \pi}2\right) \beta \xi^2] } , 
\end{equation}
As before $\beta$ and $\gamma$ depend on $\xi$ and we have  $ \beta= \sqrt{ \omega  - \xi^2}$ and $\gamma = \sqrt{ \frac{\omega}2 - \xi^2}$. 
As we do not need this  explicit representation   of $m_\omega$ for $\omega \neq 0$, we do not want to give the separate steps of the calculation. Now we  consider
the case $\omega = 0$. A fundamental solution of the differential equation 
$$ 	 A(\xi) \hat u(\xi,\cdot )  = 0 $$
is  given by the functions 
\begin{align*}
 	 v^1(x_2) &:= \left(\begin{array}{c} \xi \\ \mathrm{i} \xi  \end{array} \right) e^{-\xi x_2};
  	& v^3(x_2) &:= \left[x_2 \left(\begin{array}{c} \xi \\  \mathrm{i} \xi  \end{array} \right) +  
  	\left( \begin{array}{c}  0\\[4pt]   3 \mathrm{i} \end{array}\right)\right] e^{-\xi x_2};\\[8pt]
 	 v^2(x_2) &:= \left(\begin{array}{c} \xi  \\ - \mathrm{i} \xi \end{array} \right) e^{\xi x_2}; 
  	&v^4(x_2) &:= \left[x_2 \left( \begin{array}{c} \xi  \\ - \mathrm{i} \xi \end{array} \right) + 
  	 \left( \begin{array}{c}  0 \\[4pt]  3 \mathrm{i} \end{array}\right) \right] e^{ \xi x_2} .
\end{align*}
For the matrix $L(\xi,0)$ we obtain
\begin{equation}
	L(\xi,0) = \begin{pmatrix} 
	-2 \xi^2 \mathrm{e}^{- \xi \frac\pi2} & 	2 \xi^2 \mathrm{e}^{\xi \frac\pi2} &	(-2 \xi	- \pi \xi^2) \mathrm{e}^{-\xi \frac\pi2} &	(-2 \xi + \pi \xi^2) \mathrm{e}^{\xi\frac\pi2} \\[5pt]	
	- 2 \mathrm{i} \xi^2  \mathrm{e}^{- \xi \frac\pi2  } & 	- 2 \mathrm{i} \xi^2  \mathrm{e}^{\xi \frac\pi2  }
	&	- \mathrm{i} ( 4 \xi + \pi \xi^2 ) \mathrm{e}^{- \xi \frac\pi2  } & \mathrm{i} (4 \xi - \pi \xi^2) \mathrm{e}^{\xi \frac\pi2 }  \\[5pt]
		-2 \xi^2 	& 	2 \xi^2 & 	-2 \xi & 	-2 \xi \\[5pt]
	\mathrm{i} \xi & 	- \mathrm{i} \xi  & 	3\mathrm{i} &	3\mathrm{i} 
\end{pmatrix} . \end{equation}
Calculating explicitly   its inverse  we obtain for  $\xi \in \RR$
\begin{align*}
	 L(\xi,0)^{-1} &= 
	\begin{pmatrix} \makebox[1.4cm]{$\cdots$} & \makebox[1.4cm]{$\cdots$} & \makebox[1.4cm]{$\cdots$} &   
	\frac{\mathrm{i} \mathrm{e}^{ \pi\xi } ( - 2 + 2 \mathrm{e}^{ \pi\xi } + 2 \pi \xi + \pi^2 \xi^2) }
	{4 \xi ( - 1+ \mathrm{e}^{2\pi \xi} + 2 \mathrm{e}^{\pi \xi} + 2 \pi \xi \mathrm{e}^{\pi\xi})} \\[16pt]
	\cdots & \cdots & \cdots &   	\frac{\mathrm{i}  (  2 + 2 \mathrm{e}^{ \pi\xi }  ( -2 - 2\pi \xi + \pi^2 \xi^2) )}
	{4 \xi ( - 1+ \mathrm{e}^{2\pi \xi} + 2 \mathrm{e}^{\pi \xi} + 2 \pi \xi \mathrm{e}^{\pi\xi})} \\[16pt]
	\cdots & \cdots & \cdots &   - \frac{\mathrm{i} \mathrm{e}^{ \pi\xi } ( - 1 + \mathrm{e}^{ \pi\xi } + \pi \xi ) }
	{2 ( - 1+ \mathrm{e}^{2\pi \xi} + 2 \mathrm{e}^{\pi \xi} + 2 \pi \xi \mathrm{e}^{\pi\xi})} \\[16pt]
	\cdots & \cdots & \cdots &   - 	\frac{\mathrm{i}  ( -1  + \mathrm{e}^{ \pi\xi }  ( 1 + \pi \xi ) )}
	{2 ( - 1+ \mathrm{e}^{2\pi \xi} + 2 \mathrm{e}^{\pi \xi} + 2 \pi \xi \mathrm{e}^{\pi\xi})}  
	\end{pmatrix} .
\end{align*}
For the function $\hat u$ we have 
$$  \hat u (\xi, x_2) = \hat g(\xi) \cdot  \begin{pmatrix} v^1 ( x_2 ) \\[15pt]  v^2 ( x_2 ) 
	 \\[15pt] v^3 ( x_2 )  \\[15pt]  v^4 ( x_2 )  \end{pmatrix}^T \cdot  \begin{pmatrix}
	\frac{\mathrm{i} \mathrm{e}^{ \pi\xi } ( - 2 + 2 \mathrm{e}^{ \pi\xi } + 2 \pi \xi + \pi^2 \xi^2) }
	{4 \xi ( - 1+ \mathrm{e}^{2\pi \xi} + 2 \mathrm{e}^{\pi \xi} + 2 \pi \xi \mathrm{e}^{\pi\xi})} \\[16pt]
	\frac{\mathrm{i}  (  2 + 2 \mathrm{e}^{ \pi\xi }  ( -2 - 2\pi \xi + \pi^2 \xi^2) )}
	{4 \xi ( - 1+ \mathrm{e}^{2\pi \xi} + 2 \mathrm{e}^{\pi \xi} + 2 \pi \xi \mathrm{e}^{\pi\xi})} \\[16pt]
	  - \frac{\mathrm{i} \mathrm{e}^{ \pi\xi } ( - 1 + \mathrm{e}^{ \pi\xi } + \pi \xi ) }
	{2 ( - 1+ \mathrm{e}^{2\pi \xi} + 2 \mathrm{e}^{\pi \xi} + 2 \pi \xi \mathrm{e}^{\pi\xi})} \\[16pt]
	 - 	\frac{\mathrm{i}  ( -1  + \mathrm{e}^{ \pi\xi }  ( 1 + \pi \xi ) )}
	{2 ( - 1+ \mathrm{e}^{2\pi \xi} + 2 \mathrm{e}^{\pi \xi} + 2 \pi \xi \mathrm{e}^{\pi\xi})} \end{pmatrix}  . 
$$
In particular, we observe that $\hat u (\xi)$ does not have any singularities for $\xi \in \RR$. 
Since $ R(\xi,0) = \begin{pmatrix} 2 \mathrm{i} \xi^2 & 2 \mathrm{i} \xi^2 & 4 \mathrm{i} \xi & -4 \mathrm{i} \xi \end{pmatrix} $, 
we have  
\begin{align*}
	- 2 \partial_2 \hat u_2 (\xi  ,0) = m_0(\xi) \hat g(\xi) ,
\end{align*}
where 
\begin{align}
	m_0 (\xi) := \xi \cdot \frac{\cosh (\pi \xi) - \left( 1 + \frac{\pi^2 \xi^2}{2}\right) }{ \sinh(\pi \xi) + \pi \xi } . 
\end{align}
Until now we did not use the symmetry condition corresponding to the space  $\mathcal H_2$. 
Imposing this  additional symmetry we observe that  that $\hat u (\xi, \cdot) \in h_{2+}$ for almost  every $\xi \in \RR$,  and thus, 
\begin{equation*}\label{eq:relation_bdry_data}
	\left(\begin{array}{c} 
  		\mathrm{e}^{\mathrm{i} \beta \frac{\pi}2} - 1 \\[4pt]
  		\mathrm{e}^{- \mathrm{i} \beta \frac{\pi}2} - 1 \\[4pt]
  		\frac{\xi}\gamma \Bigr(\mathrm{e}^{\mathrm{i} \gamma \frac{\pi}2} - 1\Bigr)\\[6pt]
  		-\frac{\xi}\gamma \Bigr(\mathrm{e}^{- \mathrm{i} \gamma \frac{\pi}2} - 1\Bigr) 
	\end{array} \right)^T a(\xi, \omega ) = 0 
\end{equation*}
if  $\omega \neq 0$. An easy calculation shows that this condition is always satisfied.
This is due to the fact that the only solution of \eqref{eq:solution_operator1}-\eqref{eq:solution_operator3} in $\mathcal H_{1+}$ is the trivial one. 
Moreover, if $\omega \in [0, \Lambda)$ the boundary value problem \eqref{eq:solution_operator1}-\eqref{eq:solution_operator3} is 
not Fredholm. 
However, in the following lemma we observe that it becomes Fredholm by imposing the additional symmetry condition. 
\begin{lemma}\label{lemma:solution_op_elast}
Let $\omega \in \CC \backslash [\Lambda, \infty)$. Then for every $g \in H^{1/2} (\RR)$ 
there exists a unique $u \in H^1(\Omega_+;\CC^2) \cap \mathcal H_{2+}$ which solves
\eqref{eq:solution_operator1}-\eqref{eq:solution_operator3}, moreover we have $ \|u\|_{H^1(\Omega_+;\CC^2)} \le C(\omega) \| g\|_{H^{1/2}(\RR)}$ for some 
$C(\omega) > 0$ independent of $g$. 
\end{lemma}
\begin{proof} 
	Applying the Fourier transform  leads to the operator pencil
	$$ \mathfrak  A(\xi) = \begin{pmatrix} A(\xi) - \omega   \\ B_1 (\xi) \\ \vdots \\  B_4(\xi) \end{pmatrix} : H^2(I_+;\CC^2) 
	\to \begin{array}{c} L_2(I_+; \CC^2) \\ \oplus \\ \CC^4 \end{array} , $$
	where 
		$$ A(\xi) = \begin{pmatrix} 2 \xi^2 - \partial_2^2 & - \mathrm{i} \xi \partial_2 \\  - \mathrm{i} \xi \partial_2 &
	\xi^2 - 2 \partial_2^2 \end{pmatrix}  $$
	and 
	\begin{align*}
		B_1(\xi)u &= \partial_2 u_1 \left(\frac\pi2\right) + \mathrm{i} \xi u_2  \left(\frac\pi2\right) , & 
		B_2(\xi)u &= 2 \partial_2 u_2 \left(\frac\pi2\right) , \\
		B_3(\xi) u &= \partial_2 u_1 (0) + \mathrm{i}  \xi u_2 (0), &
		B_4(\xi) u &= u_2(0) . 
	\end{align*}
	It is well known that  for every $\xi \in \RR$ the operator $\mathfrak A(\xi)$ is a Fredholm operator with Fredholm index
	 $0$ as it corresponds to a self-adjoint problem
	on a bounded domain, cf.\ \cite[Chapter 4]{McLean} or \cite[Chapter 3]{KoMaRo}. We consider its restriction 
	$$ \mathfrak  A(\xi)|_{h_2+ } : H^2(I_+;\CC^2)  \cap h_{2+}
	\to h_{2+} \oplus \CC^4  . $$
	A short calculation shows that $\mathfrak  A(\xi)|_{h_2+ }$ is also Fredholm with  Fredholm index $0$. For $\omega \in \CC \backslash [\Lambda ,
	\infty)$ we have  $\ker (\mathfrak  A(\xi)|_{h_2+ }) = \{0\}$ for all $\xi \in \RR$, and thus, $\mathfrak  A(\xi)|_{h_2+ }$
 	is invertible for all $\xi \in \RR$. 
	Adapting slightly the proof of  \cite[Theorem 5.3.2]{KoMaRo} the invertibility of the restricted pencil shows that the Poisson problem is uniquely solvable for  $g \in H^{1/2}( \RR)$.
\end{proof}
Let $\omega \in \CC \backslash [\Lambda,\infty)$. We denote  by  $K_\omega : H^{1/2}(\RR) \to H^1(\Omega_+;\CC^2) \cap \mathcal H_{2+}$ the 
Poisson operator which maps the boundary value $g$ on the solution  $u$ of the Poisson problem \eqref{eq:solution_operator1}-\eqref{eq:solution_operator3}. 
\begin{lemma}\label{lemma:mapping_prop_K_D}
	Let $\omega \in \CC \backslash [\Lambda, \infty)$ and  $s \in \RR$. Then the operator $K_\omega$ 
	extends to a bounded linear mapping  $K_\omega : H^s(\RR) \to H^{s+1/2}(\Omega_+;\CC^2)$.  
\end{lemma}
The proof follows easily by using a variant of  \cite[Theorem 5.3.2]{KoMaRo} and   restricting  ourselves to functions $u$  which are contained in $\mathcal H_{2+}$. 
The  Dirichlet-to-Neumann operator 
$D_\omega : H^{1/2}(\RR) \to H^{-1/2}(\RR)$ is defined by 
$$  \widehat{D_\omega g} (\xi) := \hat g (\xi) \cdot m_\omega(\xi) .$$
It satisfies 
$$ D_\omega : H^{s}(\RR) \to H^{s-1} (\RR)  \qquad \text{for any } s \in \RR , $$
cf.\ also \cite[Theorem 5.3.2]{KoMaRo}. 
The following lemmas give a variational characterisation of the Poisson and  the Dirichlet-to-Neumann operators. Their proof is based on a simple integration by parts argument after applying the Fourier transform
in the horizontal direction. 
\begin{lemma}\label{lemma:var_K_omega_elast}
	Let $\omega \in \CC \backslash  [\Lambda, \infty)$ and $g \in H^{1/2}(\RR)$. For $u \in H^1(\Omega_+;\CC^2) \cap \mathcal H_{2+}$ the following conditions
	are equivalent:
	\begin{enumerate}
		\item $K_\omega g = u$. 
		\item The function $u$ satisfies $u_2|_{\RR \times \{0\}} = g$ and  for all $v \in H^1(\Omega_+; \CC^2) \cap \mathcal H_{2+} $ with 
		$v_2|_{\RR \times \{0\}} = 0$ we have 
		\begin{equation}  \int_{\Omega_+} \scal{E (u)}{\Sigma(v)}_{2 \times 2}  \; \mathrm{d} x = \omega \scal{u}{v}_{\Omega_+} . \end{equation}
\end{enumerate}
\end{lemma}

\begin{lemma}\label{lemma:var_D_omega_elast}
	Let $\omega \in \CC \backslash [\Lambda, \infty)$ and   $g \in H^{1/2}(\RR)$,  $u = K_\omega g \in H^1(\Omega_+;\CC^2) \cap \mathcal H_{2+}$.
 	Then for $h \in H^{-1/2}(\RR)$ the following two assertions are equivalent:
	\begin{enumerate}
 		\item $D_\omega g = h$.
 		\item For all $v \in H^1(\Omega_+; \CC^2) \cap \mathcal H_{2+} $  we have 
  		\begin{equation} \int_{\Omega_+} \scal{E (u)}{\Sigma(v)}_{2 \times 2}  \; \mathrm{d} x = \omega \scal{u}{v}_{\Omega} + \scal{h}{v_2(\cdot,0)}_{\RR} .\end{equation}
	\end{enumerate}
\end{lemma}
Next we prove  a perturbation formula for the Dirichlet-to-Neumann operator with respect to the spectral parameter
$\omega$, which will be essential for the proof of Theorem \ref{th:main_2D}. 
\begin{theorem}\label{th:perturb-to-N}
	Let $\omega,  \eta \in \CC \backslash [\Lambda, \infty)$. Then the following identities hold true:
	\begin{enumerate}
		\item $K_\omega = ( I + (\omega - \eta) ( A_{\varnothing +}^{(2)}  - \omega)^{-1}) K_\eta$;
		\item $D_\omega^* = D_{\overline{\omega}} $;
		\item $D_\omega = D_{\overline{\eta}} - ( \omega - \overline{\eta})   K^*_\eta K_\omega$;
		\item $D_\omega =  D_{\overline{\eta}} -   ( \omega - \overline{\eta})  K^*_\eta \Bigr( I + (\omega - \eta) 
		( A_{\varnothing +}^{(2)} -  \omega)^{-1}  \Bigr) K_\eta $.
	\end{enumerate}
	We note that  $K_\omega \in \mathcal L( H^{-1/2}(\RR) ,  L_2(\Omega_+;\CC^2))$, and thus, its adjoint satisfies 
$K_\omega^* \in \mathcal L( L_2(\Omega_+;\CC^2)  , H^{ 1/2}(\RR) )$.
\end{theorem}
A more general assertion is given in  \cite[Proposition 2.6]{BehLanger07} in the context of boundary triplets.
\begin{proof}
	The proof is based on variational arguments. We note that  (4) follows by combining assertions (1) and (3).
	Let $g \in H^{1/2}(\RR)$ and $$u := K_\eta g + (\omega - \eta) ( A_{\varnothing+}^{(2)}  - \omega)^{-1}) K_\eta g.$$
	Then $u(x,0) = g$ and for $v \in H^1(\Omega_+;\CC^2) \cap \mathcal H_{2+}$, $v_2|_{\RR \times \{0\}}  = 0$ we have 
	\begin{align*}
		\int_{\Omega_+} \scal{E(u)}{\Sigma(v) }_{2 \times 2}  \; \mathrm{d} x 
		&= \eta \scal{K_\eta g}{v}_{\Omega_+} +   (\omega - \eta) \cdot a_{0+} [( A_{\varnothing +}^{(2)} - \omega)^{-1} K_\eta g, v]\\
		&= \eta \scal{K_\eta g}{v}_{\Omega_+} + (\omega - \eta) \cdot \scal{A_{\varnothing+}^{(2)} ( A_{\varnothing+}^{(2)} - \omega)^{-1} K_\eta g}{v}_{\Omega_+} \\[4pt]
		&= \omega\scal{K_\eta g}{v}_{\Omega_+}
		+ \omega (\omega - \eta) \scal{(A_{\varnothing+}^{(2)} - \omega)^{-1} K_\eta g}{v}_{\Omega_+} \\[4pt]
		&= \omega \scal{u}{v}_{\Omega_+} . 
	\end{align*}
	Now Lemma \ref{lemma:var_D_omega_elast} implies the first assertion. Next we prove assertions (2) and (3). For $g , h \in H^{1/2}(\RR)$ and 
	$u := K_\omega g$, $v := K_\eta h$ we have 
	$$ \scal{D_\omega g}{h}_\RR - \scal{g}{D_\eta h }_\RR
		= (\omega - \overline{\eta}) \scal{u}{v}_{\Omega_+} = (\omega - \overline{\eta}) \scal{g }{K_\omega^* K_\eta v}_{\Omega_+} . $$
	This proves (2) and (3). 
\end{proof}
Now we  return to the mixed problem and  introduce the  truncated operator acting only on functions, which are supported on $\overline{\Gamma}$. 
We define
\begin{align}\label{def:Hs_0}
	\tilde H_0^{1/2}(\Gamma)  &:= \{ g \in H^{1/2}(\RR) :  \mathrm{supp}(g) \subseteq   \overline{\Gamma}  \} , \\
	H^{- 1/2} (\Gamma) &:=  \{ g \in (C_c^\infty(\Gamma))' :  \exists G \in H^{-1/2}(\RR) \text{ such that }  g = G|_{\Gamma} \} .
	\label{def:Hs} 
\end{align}
Here $C_c^\infty(\Gamma)$ is the space of  smooth functions with compact support contained in  $\Gamma $; its dual   $(C_c^\infty(\Gamma))'$ is the space of distributions on $\Gamma$. 
We note that $\tilde H_0^{1/2}(\Gamma)$ is a closed subspace of distributions on $\RR$ whereas
$H^{-1/2}(\Gamma)$ is a subspace  of distributions on  $\Gamma$. The latter
space may be identified with the quotient space
$$ \raisebox{0.5ex}{\ensuremath{H^{-1/2}(\RR)}}
\ensuremath{\mkern-0.3mu}/\ensuremath{\mkern-0.3mu}
\raisebox{-0.5ex}{\ensuremath{\tilde H^{-1/2}_0(\RR \backslash \overline{\Gamma})}} , $$
where $\tilde H^{-1/2}_0(\RR \backslash \overline{\Gamma})$ contains, by definition, those distributions 
in $H^{-1/2}(\RR)$ which have support in $\RR \backslash \Gamma$. We endow the spaces in \eqref{def:Hs_0} and 
\eqref{def:Hs} with their natural topology, i.e., $\tilde H^{1/2}_0(\Gamma)$ carries  the subspace topology of $H^{1/2}(\RR)$ and 
$H^{-1/2}(\Gamma)$ has the quotient topology. Note that  we may identify $\tilde H^{1/2}_0(\Gamma)$ 
with the subspace of $L_2(\Gamma)$ which consists of 
those functions whose extension by $0$ yields an element of $H^{1/2}(\RR)$. Furthermore, the 
space  $\tilde H^{1/2}_0(\Gamma)$ is an isometric realisation of the (anti-)dual of
$H^{-1/2}(\Gamma)$ and vice-versa. The dual pairing is given by the expression
\begin{equation}\label{eq:dual_pairing}
	\scal{g}{h}_{\Gamma} := \scal{G}{h}_{\RR}  ,\qquad g \in H^{-1/2}(\Gamma) , \; h \in \tilde H^{1/2}_0(\Gamma) ,
\end{equation}
where $G \in H^{-1/2}(\RR)$ denotes any extension of $g$, cf.\ \cite[Theorem 3.14]{McLean}. In particular \eqref{eq:dual_pairing} is independent of the chosen extensions $G$ which
is due to the fact, that  $C_c^\infty(\Gamma)$ is a dense subset of $\tilde H^{1/2}_0(\Gamma)$, cf.\ \cite[Theorem 3.29]{McLean}. 
Thus, the  domain of the quadratic form $a_{\Gamma+}$ may 
be rewritten  as  follows
\begin{equation}
    D[a_{\Gamma+}] =  \left\{ u \in H^{1}(\Omega_+;\CC^2 ) :   u_2|_{\RR \times \{0\} } \in \tilde H^{1/2}_{0}(\Gamma) \right \} .
\end{equation}
Then we define  the truncated Dirichlet-to-Neumann operator 
\begin{equation}
	D_{\Gamma,\omega} : \tilde H^{1/2}_{0}(\Gamma) \to H^{-1/2} (\Gamma), \qquad  
	D_{\Gamma,\omega} := r_\Gamma D_\omega e_\Gamma  , \end{equation}
where $ r_\Gamma : H^{-1/2}(\RR) \to H^{-1/2}(\Gamma)$ is the restriction operator and $e_\Gamma : \tilde H^{1/2}_{0}(\Gamma) \to H^{1/2}(\RR)$
is the embedding. Identifying $ \tilde H^{1/2}_{0}(\Gamma) $
with a subspace of $L_2(\Gamma)$, the operator $e_\Gamma$ is simply extension by $0$. 
\begin{theorem}
 	Let $\omega \in \CC \backslash [\Lambda , \infty)$. Then
	\begin{equation} \mathrm{dim}\; \mathrm{ker} (A_{\Gamma}^{(2)} - \omega) = \mathrm{dim}\; \mathrm{ker}  ( D_{\Gamma, \omega } ) . \end{equation}
\end{theorem}
The proof is an easy consequence of the variational characterisation of the Dirichlet-to-Neumann operator in Lemma \ref{lemma:var_D_omega_elast} and the duality of $\tilde H^{1/2}_{0}(\Gamma) $ and $H^{-1/2} (\Gamma)$. 
We note that 
$$ D_{\Gamma, \omega} : \tilde H^{1/2}_{0}(\Gamma) \to \tilde H^{1/2}_{0}(\Gamma)^*  ,   $$
and thus, the operator $D_{\Gamma, \omega}$ may be completely described  by its  sesquilinear form
\begin{align}\label{def:d_ell}
	d_{\Gamma, \omega} [g,h] :=  \scal{  D_{\Gamma,\omega} g}{h}_{H^{-1/2}(\Gamma) ,\tilde H^{1/2}_{0}(\Gamma)} 
\end{align}
where $g,h \in D[d_{\Gamma,\omega}] :=  \tilde H^{1/2}_0(\Gamma)$. Using the dual 
pairing \eqref{eq:dual_pairing} we have for $g,h \in \tilde H^{1/2}_0(\Gamma)$  
\begin{align}\label{eq:d_ell}
	d_{\Gamma, \omega}[g,h] &= \scal{D_\omega g}{h}_{\RR} = \int_{\RR} m_\omega(\xi)  \cdot \hat g(\xi) \; \overline{\hat h(\xi)}   \; \mathrm{d} \xi 
\end{align}
since $D_{\omega} g \in H^{-1/2}(\RR)$ is an extension of $D_{\Gamma, \omega }g \in H^{-1/2}(\Gamma)$.
Lemma \ref{lemma:var_D_omega_elast} implies that 
$$ d_{\Gamma,\omega}[g,h]  = \scal{D_\omega g}{h}_{\RR } =  a_{\Gamma_+} [ K_\omega g , K_\omega h] -
  \omega \scal{ K_\omega g}{K_\omega h}_{\Omega_+}  . $$
We note that the expression \eqref{eq:d_ell} is  independent of $\Gamma$; in particular   the dependence on $\Gamma$ enters 
as a constraint  on the support of the functions $g$ and $h$. 
\begin{lemma}\label{lemma:form_DtoN}
	Let $\omega \in \CC \backslash [\Lambda, \infty)$. Then  $d_{\Gamma,\omega}$ defines a closed sectorial  form 
	in $L_2(\Gamma)$. The associated m-sectorial  operator
	is the restriction of $D_{\Gamma,\omega} $ to  the operator domain
  	\begin{equation}  X_{\Gamma,\omega}  := \left\{ g \in \tilde H^{1/2}_0 (\Gamma) :   D_{\Gamma,\omega} g \in L_2(\Gamma) \right\} . \end{equation}
	If $\omega$ is real, then the associated  operator is self-adjoint. 
\end{lemma}
\begin{proof}
  	For  $g \in \tilde H^{1/2}_0(\Gamma) $ we use the identity
	\begin{align*}
		d_{\Gamma,\omega}[g] & = d_{\Gamma,\omega}[g,g] = \scal{D_\omega g}{g}_{\RR } =  a_{\Gamma_+}[ K_\omega g] - \omega \| K_\omega g\|_{L_2(\Omega_+;\CC^2)}^2 .
	\end{align*}
	Let $u  := K_\omega g$. The mapping property of the Poisson operator $K_\omega$ implies 
  	$$ |d_{\Gamma, \omega} [g]| \le c   \| u \|^2_{H^1(\Omega_+;\CC^2)} \le c_1 \| g \|^2_{\tilde H^{1/2}_0(\Gamma)}  . $$
	Again the mapping properties of the Poisson operator and the trace theorem imply that 
  	\begin{align*}
   		\Re( d_{\Gamma, \omega} [g]) &=  \| u \|^2_{H^1(\Omega_+;\CC^2)}  - \Re  (\omega - 1)   \| u \|_{L_2(\Omega_+;\CC^2)}^2 \\
  		&\ge c_1   \|g\|^2_{\tilde H^{1/2}_0(\Gamma)}  -  c_2 \| g\|^2_{H^{-1/2}(\RR)}
  		\ge 
  			c_1 \|g\|^2_{\tilde H^{1/2}_0(\Gamma)}  - c_3 
  			\| g\|^2_{L_2(\Gamma)} . 
  	\end{align*}
  	Thus, $d_{\Gamma, \omega}$ is closed and accretive. Analogously it follows that $d_{\Gamma,\omega}$ is sectorial.  Moreover, 
 	a short calculation implies that the associated m-sectorial operator is exactly the restriction 
	of $D_{\Gamma, \omega}$ to $X_{\Gamma,\omega}$. Finally, if $\omega \in \RR$ then $d_{\Gamma, \omega}[g] \in \RR$, and thus, the associated operator is self-adjoint. 
\end{proof}
As a particular consequence of Lemma \ref{lemma:form_DtoN} we obtain that 
$$ \mathrm{ker} ( D_{\Gamma,\omega} ) \subseteq X_{\Gamma,\omega} . $$ 
Since the spectrum of $A_{\Gamma+}^{(2)}$ is a subset of the real axis  we may restrict ourselves the case $\omega \in \RR$ and we can apply methods from spectral theory
to determine whether zero is an eigenvalue of $D_{\Gamma ,\omega} $ or not. 
Note that  $D[d_{\Gamma,\omega}] = \tilde H^{1/2}_0(\Gamma)$ is compactly embedded\footnote{cf. \cite[Theorem 3.27]{McLean}.} into $L_2(\Gamma)$, and thus,
the spectrum of the m-sectorial realisation consists of a discrete set of eigenvalues only accumulating at infinity. Moreover, the proof of Lemma  \ref{lemma:form_DtoN} and \cite[Theorem 2.34]{McLean}
imply that the original operator $D_{\Gamma, \omega} : \tilde H^{1/2}_0(\Gamma) \to H^{-1/2}(\Gamma)$ is a Fredholm operator with
zero index. 

Naturally the above considerations remain valid if we replace $\Gamma$ by $\Gamma_\ell := \ell \cdot \Gamma$. 
We note that the operators $D_{\Gamma_\ell, \omega}$ are each acting  in a different Hilbert space for different $\ell > 0$.
To obtain a family of operators acting in the same space we introduce the scaling operators
\begin{equation}
 T_{\ell}  :  L_2(\Gamma)  \to  L_2(\Gamma_\ell), \qquad (T_\ell g)(x) = 
	\ell^{-1/2} \cdot g (x/\ell)  
\end{equation}
and note that the operator $T_\ell$ bijectively maps  $\tilde H^{1/2}_0(\Gamma)$ into $\tilde H^{1/2}_0(\Gamma_\ell)$. Let 
\begin{align}
	\mathcal Q(\ell, \omega) : \tilde H^{1/2}_0(\Gamma) \to H^{-1/2}(\Gamma), \qquad \mathcal Q(\ell, \omega) := T_\ell^* D_{\Gamma_\ell, \omega}  T_\ell 
\end{align}
with the associated sesquilinear form 
\begin{align}
	q(\ell, \omega ) [g,h] := d_{\Gamma_\ell,\omega} [T_\ell g, T_\ell h ]  , \qquad g , h \in D[q(\ell, \omega ) ] := \tilde H^{1/2}_{0}(\Gamma) . 
\end{align}
Then for $\ell > 0$ and $\omega \in \CC \backslash [\Lambda,  \infty)$ we have 
$$ \mathrm{dim}\; \mathrm{ker} (A_{\Gamma_\ell }^{(2)} - \omega) = \mathrm{dim}\;  \mathrm{ker}(  D_{\Gamma_\ell, \omega} ) = \mathrm{dim}\; \mathrm{ker}(\mathcal Q(\ell, \omega) ) . $$ 
From \eqref{eq:d_ell} we obtain for  $g, h \in \tilde H^{1/2}_0(\Gamma)$ 
\begin{align*}
	q(\ell, \omega )[g,h] = d_{\Gamma_\ell, \omega} [T_\ell g, T_\ell h] &
  	= \ell \int_\RR  m_\omega(\xi)  \cdot \hat g(\ell \xi ) \; \overline{\hat h(\ell \xi)}  \; \mathrm{d} \xi 
  	= \int_\RR   m_\omega\left(\xi/\ell \right) \cdot \hat g(\xi) \; \overline{\hat h(\xi)}  \; \mathrm{d} \xi .
\end{align*}

Next we describe the behaviour of the form $q(\ell,\omega)$ as $\ell \to 0$ and $\omega \to \Lambda$. This  asymptotic expansion will  represent
the principal tool for the proof of Theorem \ref{th:main_2D}. 
Let  $\mathcal Q_0 : \tilde H_0^{1/2}(\Gamma) \to H^{-1/2}(\Gamma)$, 
\begin{align}
	\scal{\mathcal Q_0 g}{h}_{\Gamma} := q_0[g,h] := \int_{\RR} |\xi| \cdot \hat g(\xi) \; \overline{\hat h (\xi)}  \; \mathrm{d} \xi .
\end{align}
We note  that   $\mathcal Q_0 : \tilde H^{1/2}_0(\Gamma) \to  H^{-1/2}(\Gamma)$ is 
a Fredholm operator with Fredholm index $0$, which follows from \cite[Theorem 2.34]{McLean}. Hence,  $\mathcal Q_0$ is invertible since it has trivial kernel. Indeed, the identity $\mathcal Q_0 g = 0$ implies that 
$$ 0= \scal{\mathcal Q_0 g}{g}_{\Gamma} = \int_{\RR} |\xi| \cdot | \hat g(\xi)|^2  \; \mathrm{d} \xi , $$
and thus, $g = 0$.
Furthermore, we denote by  $P_\pm$    the projection onto the subspace in $L_2(\Gamma)$ spanned by the functions 
\begin{align}
	\Phi_\pm (x_1) := \mathrm{e}^{\pm \mathrm i \varkappa x_1}  
\end{align}
and let 
$\psi_{\pm \kappa}  = (\psi_{\pm  \varkappa,1}, \psi_{\pm \varkappa,2})^T \in  L_2(I_+;\CC^2)$
be chosen such that 
\begin{align}
	A^{(2)}_{\varnothing+}(\pm \varkappa) \psi_{\pm \varkappa}= \Lambda \psi_{\pm  \varkappa} \qquad \text{and} \qquad 
		\|\psi_{\pm \varkappa} \|_{ L_2(I_+;\CC^2)} = 1 , 
\end{align}
cf.\ also Formula  \eqref{eq:eigenfuncion_elast}, where a non-normalised eigenfunction for the unitarily equivalent
operator $A_{\varnothing}^{(2)}$ is given. 
\begin{theorem}\label{th:asymptotics_q}
	 There exists $\ell_0 > 0$ and $\varepsilon > 0 $  such that  for all $\ell \in (0,\ell_0)$ and $| \omega - \Lambda| < \varepsilon $ 
	the following expansion holds true 
	\begin{align}
 		\mathcal Q(\ell, \omega) &= \frac1\ell \mathcal Q_0 -  \frac{4 |\Gamma| \cdot  |\partial_2 \psi_{\varkappa, 2} (0)|^2}{\sqrt{\Lambda -\omega} \cdot \sqrt{2 \zeta_1 ''(\varkappa) }}
		 \;  T_\ell^* \bigr( P_+ + P_- \bigr)  T_\ell +  R(\ell,\omega) . 
	\end{align}
	Here  $|\Gamma|$ is the Lebesgue measure of $\Gamma$ and the remainder  satisfies the following estimate
	$$  \sup\{ \| R(\ell,\omega) \|_{\mathcal L(L_2(\Gamma))} :  \ell \in (0,\ell_0) \; \wedge \; |\omega - \Lambda| < \varepsilon \} < \infty . $$
\end{theorem}
The next section is devoted to the proof of Theorem \ref{th:asymptotics_q}. 
\subsection{The proof of Theorem \ref{th:asymptotics_q}}
For the proof we use the perturbation formula for the  Dirichlet-to-Neumann operator $D_\omega$ in   Theorem \ref{th:perturb-to-N} (4), which we apply for $\eta=0$. We obtain   
$$  D_\omega = D_0 - \omega K_0 ( I + \omega (A_{\varnothing+}^{(2)}  - \omega)^{-1}) K_0^*.  $$
Thus, for  $g,h \in \tilde H^{1/2}_0(\Gamma)$ we have 
\begin{align*}
	  \scal{\mathcal Q(\ell, \omega) g}{h}_\Gamma &=  q(\ell, \omega ) [g,h]   = d_{\Gamma_\ell, \omega} [T_\ell g, T_\ell h ] = \scal{D_\omega T_\ell g}{T_\ell hg}_\RR \\
	  &=   \scal{D_0  T_\ell g}{T_\ell h}_\RR -  \omega \scal{( I + \omega (A_{\varnothing+}^{(2)} - \omega)^{-1}) K_0 T_\ell g }{K_0 T_\ell h}_{\Omega_+}  \\
	  &=  q(\ell,0) [g, h] -  \omega  \scal{ ( I + \omega (A_{\varnothing+}^{(2)} - \omega)^{-1}) K_0 T_\ell g }{K_0 T_\ell h}_{\Omega_+} . 
\end{align*} 
Recall that 
$$ q(\ell, 0) [g,h] =  \int_{\RR} m_0(\xi/\ell) \cdot \hat g (\xi) \; \overline{\hat h(\xi)}  \; \mathrm{d} \xi  $$
with  
$$ m_0 (\xi) =  \xi \cdot \frac{\cosh (\pi \xi) - 1 - \frac{\pi^2 \xi^2}{2}}{ \sinh(\pi \xi) + \pi \xi } . $$
Using the estimate  $  m_0 (\xi) = |\xi| + \mathcal O ( 1 )$ we obtain 
\begin{align*}
	q(\ell,0)[g,h]  &= \frac{1}{\ell} \int_{\RR} |\xi| \cdot \hat g (\xi) \; \overline{\hat h(\xi)}  \; \mathrm{d} \xi  + \scal{ \tilde R (\ell) g}{h}_{\Gamma} 
	= \frac1\ell \scal{\mathcal Q_0 g}{h}_{\Gamma} + \scal{\tilde R (\ell) g}{h}_{\Gamma} ,  
\end{align*}
where 
$$  \sup\{ \| \tilde R (\ell) \|_{\mathcal L(L_2(\Gamma))} : \ell \in (0,1)  \} < \infty . $$
Now we treat the resolvent term. To this end we need to understand the behaviour of the solutions
of the Rayleigh-Lamb equation \eqref{eq:Rayleigh-Lamb}  as  $\omega \to \Lambda$. 
\begin{lemma}\label{lemma:sing_elastic}
	Let $\beta = \sqrt{\omega - \xi^2}$, $\gamma = \sqrt{\frac{\omega}2 - \xi^2}$. 
	There exist  $\Theta  > 0$ and $\varepsilon > 0  $ such that 
	for all $| \omega -  \Lambda| < \varepsilon $ the Rayleigh-Lamb equation 
	$$  \frac{\sin \left(\beta \frac{\pi}2\right)}{\beta}  \cos \left(\gamma \frac{\pi}2\right)
	\gamma^2 + \cos \left(\beta \frac{\pi}2\right) \frac{\sin \left(\gamma \frac{\pi}2\right)}{\gamma} \xi^2  = 0 $$
	has exactly four  solutions  in the infinite strip $\RR + \mathrm{i} [-\Theta,\Theta]$. For  $\omega \in \CC \backslash [\Lambda, \infty)$, two of these 
	solutions have strictly positive imaginary part and two of them have  strictly negative imaginary part. 
	There  exists a holomorphic function $H$ with $H (0) =\varkappa$ and $ H'(0) = (\zeta_1''(\varkappa)/2)^{-1/2} $ such that the solutions with positive imaginary part
	are given by
	$$    \xi_{1+}(\omega) = - H( - \mathrm{i} \sqrt{\Lambda - \omega}) \qquad \text{and} \qquad \xi_{2+}(\omega) = H   ( \mathrm{i}  \sqrt{\Lambda - \omega} )  . $$
\end{lemma}
We recall that  $\zeta_1(\xi)$ is   the first eigenvalue of the operators $A_{\varnothing}^{(2)}(\xi)$ and $A_{\varnothing+}^{(2)}(\xi)$ for $\xi \in \RR$. Moreover, we have
$$ \zeta_1(\pm \varkappa) = \Lambda = \min\{ \zeta_1 (\xi) : \xi \in \RR \}  . $$
\begin{proof}
Let 
$$ \Psi(\xi, \omega) 
	:= 
      \frac{\sin \left(\beta \frac{\pi}2\right)}{\beta}  \cos \left(\gamma \frac{\pi}2\right)
      \gamma^2 + \cos \left(\beta \frac{\pi}2\right) \frac{\sin \left(\gamma \frac{\pi}2\right)}{\gamma} \xi^2 
$$
be the left-hand side of the Rayleigh-Lamb equation, where 
as before
$$ \beta = \sqrt{ \omega  - \xi^2} \qquad \text{and} \qquad \gamma = \sqrt{ \frac{\omega}2 - \xi^2}. $$ 
Expanding the sine and cosine functions into their power series shows that only powers of the square root function
with even exponent are present, and thus, $\Psi$ is holomorphic in both variables $(\xi,\omega) \in \CC^2$. We note that for values 
$\xi \in \RR$ and $\omega \in \CC$ with  $\Psi(\xi, \omega) = 0$ we necessarily have $\zeta_k (\xi) = \omega$ for some $k \in \NN$. %Next, we 
In the case $\omega = \Lambda$ the function $\RR \ni \xi \mapsto \Psi(\xi , \Lambda  )$ has exactly two zeros $\pm \varkappa$, cf.\ Section \ref{subsec:dispersion_curves} or  \cite{FoeWeidl}. 
Each  of them has multiplicity $2$, i.e., we have 
\begin{equation}\label{eq:double_pole_Phi}
	\partial_\xi \Psi(\pm \varkappa, \Lambda) = 0 \qquad \text{and} \qquad \partial_\xi^2 \Psi(\pm \varkappa, \Lambda) \neq 0 .
\end{equation}
The argument principle for holomorphic functions implies that there exist two constants $\Theta > 0$ and $\delta > 0$ 
such that $\Psi(\cdot , \omega)$ has exactly 4 zeros counted with multiplicities
in the infinite strip $\RR + \mathrm{i} [-\Theta, \Theta]$ for $| \omega - \Lambda | < \varepsilon  $. Two of these zeros are located near $\varkappa$ and the others are near $-\varkappa$. 
Note that we used that $| \Psi(\xi , \omega ) |$ tends  to infinity as $\Re (\xi) \to \pm \infty$, locally uniform in $\omega$. 

For the sake of simplicity we consider for the rest of the proof only those zeros near $\varkappa$.
We note  that  the function $\zeta_1$ is chosen such that $\Psi(\xi, \zeta_1(\xi)) = 0$ for $\xi \in \RR$. Moreover, we have  $\partial_\omega \Psi (\varkappa, \Lambda) \neq 0$.
This  follows  by an numerical calculation,  
which can be made rigorous by approximating the corresponding power series, cf.\ also the considerations in \cite{FoeWeidl}.  
Thus, there exists 
neighbourhoods $U_\varkappa$ of $\varkappa$ and $V_\Lambda$ of $\Lambda$ so that for all $(\xi, \omega) \in U_\varkappa \times V_\Lambda$ the
identity $\Psi(\xi, \omega) = 0$ holds true if and only if $\omega = \zeta_1(\xi)$. 
Note that   $\zeta_1'(\varkappa) = 0$ since 
$\varkappa$ is a global minimum and  $\zeta_1''(\varkappa) \neq 0$. As $\zeta_1$ is real analytic
there exists a neighbourhood $V_0$ around $0$ and an invertible analytic 
function $G : U_\varkappa \to V_0 $ such that 
$$ \zeta_1(\xi) = \Lambda + G(\xi)^2 , \qquad \xi \in U_\varkappa  .$$
Setting $H := G^{-1}$ we observe that  the two  zeros of $\Psi(\cdot, \omega)$ near $\varkappa$ are given by
$$ G^{-1} (  \pm \mathrm{i} \sqrt{\omega - \Lambda}) = H  (  \pm \mathrm{i} \sqrt{\omega - \Lambda}) .$$
Note that  we may choose $G$ such that $G'(\varkappa) = (\zeta_1''(\varkappa)/2)^{1/2} $, and thus, 
$H' (0) = (\zeta_1''(\varkappa)/2)^{-1/2} > 0$.  The Taylor expansion of $H$ shows that one zero has strictly positive
 imaginary part, the other strictly negative imaginary part. 
\end{proof}

Now we may give an estimate for the resolvent term. Let $g, h \in \tilde H^{1/2}_0(\Gamma)$. For ease of notation we put $g_\ell := T_\ell g$ and $h_\ell := T_\ell h$. Then 
\begin{align}	
	&\omega\scal{( I + \omega ( A_{\varnothing +}^{(2)} - \omega)^{-1} ) K_0 T_\ell g }{K_0 T_\ell h}_{\Omega+} \notag\\
	&\quad= \omega \int_{\RR} \scal{(I + \omega (A_{\varnothing +}^{(2)} (\xi) - \omega)^{-1}) K_0 (\xi) \hat g_\ell (\xi)}{K_0 (\xi) \hat h_\ell(\xi)}_{I+}  . 
	\label{eq:F(xi)}
\end{align}
In what follows we use that $K_0(\cdot ) : \CC \to H^2(I_+;\CC^2) \cap h_{2+}$ is a finitely meromorphic\footnote{For the definition of finitely meromorphic functions we refer to \cite{GohbergSigal}.} function with  singularities, which  are located at most at those points $\xi \in \CC$ which solve the Rayleigh-Lamb equation $ \Psi (\xi, 0) = 0$. In particular we may choose $\Theta > 0$ and  $\varepsilon > 0$ such that $K_0(\cdot)$ is holomorphic in $\RR + \mathrm{i} [-2 \Theta, 2 \Theta]$ and such that for all  $|\omega - \Lambda | < \varepsilon$ the Rayleigh-Lamb equation has exactly four solutions in the infinite strip $\RR + \mathrm{i} [-  2 \Theta, 2\Theta ]$.
Let 
$$ F_\omega (\xi ) :=  \scal{( I + \omega(A_{\varnothing +}^{(2)}(\xi) - \omega)^{-1} ) K_0(\xi) \hat g_\ell  (\xi) }{ 
	 	K_0(\overline{\xi}) \hat h_\ell (\overline{\xi}) ) }_{I+}  $$
be the integrand in \eqref{eq:F(xi)}. Then $F_\omega$ may be extended to a meromorphic function on the strip 
$\RR  + \mathrm i [- 2 \Theta, 2 \Theta]$ since the functions $\hat g_\ell , \hat h_\ell $ have  compact support, 
and thus, $\hat g_\ell, \hat h_\ell$ may be extended to holomorphic functions on all of $\CC$. We have 
\begin{align*}
	\int_{\RR} F_\omega(\xi)  \; \mathrm{d} \xi &= \int_{-\varkappa - }^{-\varkappa + \delta} F_\omega(\xi)   \; \mathrm{d} \xi
		+ \int_{\varkappa - \delta}^{\varkappa +  \delta} F_\omega(\xi)   \; \mathrm{d} \xi  + \left( \int_{-\infty}^{-\varkappa - \delta} + 
		\int_{-\varkappa + \delta}^{\varkappa - \delta} + \int_{\varkappa + \delta}^\infty \right) 
		F_\omega(\xi)  \; \mathrm{d} \xi  
\end{align*}
for some  $\delta > 0$. Note that
\begin{align*} 
	\left|  \int_{-\infty}^{-\varkappa - \delta}  F_\omega(\xi)  \; \mathrm{d} \xi  \right| 
	& \le   \sup_{\xi < - \varkappa - \delta} \| (I + \omega ( A_{ \varnothing +}^{(2)}(\xi) - \omega)^{-1}\| \cdot 
		\| K_0 g_\ell \|_{L_2(\Omega_+;\CC^2)}^2  \| K_0 h_\ell \|_{L_2(\Omega_+;\CC^2)}^2  \\ 
	&\quad \le C \| g  \|_{L_2(\Gamma)} \cdot \| h  \|_{L_2(\Gamma)} = \mathcal O(1) ,
\end{align*}
since  the resolvent  may be estimated by the distance of $\omega$ to the spectrum of $A_{\varnothing+}^{(2)}(\xi)$ and 
\begin{align*}
	\int_\RR \| K_0 (\xi) \hat g_\ell (\xi)\|_{L_2(I_+)}^2  \; \mathrm{d} \xi = \| K_0 g_\ell  \|_{L_2(\Omega_+;\CC^2)}^2 
\end{align*}
In the same way we may treat the integrals $\int_{-\varkappa + \delta}^{\varkappa - \delta} F_\omega(\xi)  \; \mathrm{d} \xi $
and $\int_{\varkappa + \delta }^\infty  F_\omega(\xi)  \; \mathrm{d} \xi$. Thus, 
$$ \left( \int_{-\infty}^{-\varkappa - \delta} + 
	\int_{-\varkappa + \delta}^{\varkappa - \delta} + \int_{\varkappa + \delta }^\infty \right) 
	F_\omega(\xi)  \; \mathrm{d} \xi = \mathcal O(1) . $$
To estimate the remaining integrals we consider  the following expansion of the resolvent 
\begin{align*}	
	I + \omega (A_{\varnothing+}^{(2)} (\xi) - \omega )^{-1} &=  \sum_{k=1}^\infty 
		\left( 1 + \frac{\omega }{\zeta_k(\xi) - \omega } \right) P_k (\xi)   \\
	& = \frac{\zeta_1(\xi) }{\zeta_1(\xi) - \omega}  P_1(\xi)  +   
		\sum_{k=2}^\infty \left( 1 + \frac{\omega }{\zeta_k(\xi) - \omega } \right)P_k(\xi)  , 
\end{align*}
where  $P_k(\xi)$ is the projection onto the eigenspaces $\ker (A_{\varnothing+}^{(2)}(\xi)  - \zeta_k(\xi)) $.
Note that  
$$ \biggr\|  \sum_{k=2}^\infty \left( 1 + \frac{\omega }{\zeta_k(\xi) - \omega } \right) P_k(\xi) \biggr\|_{\mathcal L (L_2( I_+))} \le 
	1 + \frac{\omega }{ \min \left\{ \zeta_2(\xi ) : \xi \in \RR \right\} - \omega  } \le \tilde c ,  $$
for $|\omega - \Lambda| < \varepsilon$. Thus, 
$$ 
\int_{\pm\varkappa - \delta}^{\pm\varkappa + \delta} F_\omega (\xi)   \; \mathrm{d} \xi = 
	\int_{\pm \varkappa - \delta}^{\pm \varkappa + \delta} 
	\frac1{\zeta_1(\xi) - \omega}  \scal{P_1(\xi) K_0(\xi) \hat g_\ell(\xi) }{ K_0(\xi) \hat h_\ell(\xi)}_{I_+}  \; \mathrm{d} \xi + \mathcal O(1) . $$ 
Now we choose two paths  $\gamma_j$, $j =1,2$,  in the complex plane which 
run around the boundaries  of the following rectangles except for the two line segments on the real axis: 
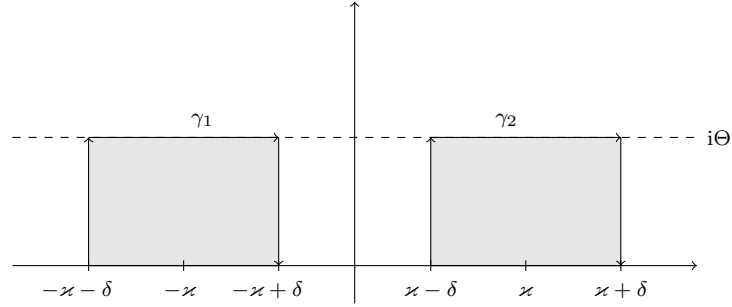
\begin{center}
\begin{tikzpicture}
\begin{scope}[font=\scriptsize]
	\draw[->]  (-4.5,0) -- (4.5,0);
	\draw[->]  (0,-0.5) -- (0,3.5);
	\draw  (-2.25,2pt) -- (-2.25,-2pt);
	\draw  (-2.3,-2pt)  node[anchor=north ] {$-\varkappa \vphantom{\delta}$}; 
	\draw  (-1,2pt) -- (-1,-2pt);
	\draw (-1.15,-2pt)  node[anchor=north ] {$-\varkappa +  \delta$}; 
	\draw  (1,2pt) -- (1,-2pt);
	\draw (1,-2pt)  node[anchor=north ] {$\varkappa -  \delta$}; 
	\draw (2.25,2pt) -- (2.25,-2pt)  node[anchor=north ] {$\varkappa \vphantom{\delta}$}; 
	\draw  (-3.5,2pt) -- (-3.5,-2pt);
	\draw (-3.65,-2pt)  node[anchor=north ] {$-\varkappa - \delta$}; 
	\draw  (3.5,2pt) -- (3.5,-2pt);
	\draw (3.5,-2pt)  node[anchor=north ] {$\varkappa +  \delta$};   
	\draw [-] (-3.5,0) -- (-1,0);
	\draw [<-] (-1,0) -- (-1,1.7);
	\draw [->] (-3.5,1.7) -- (-1,1.7);
	\draw [->] (-3.5,0) -- (-3.5,1.7);
	\draw [-] (3.5,0) -- (1,0);
	\draw [->] (1,0) -- (1,1.7);
	\draw [<-] (3.5,1.7) -- (1,1.7);
	\draw [<-] (3.5,0) -- (3.5,1.7);
	\draw (2,1.7)  node[anchor=south]  {$\gamma_2$};
	\draw (-2,1.7)  node[anchor=south]  {$\gamma_1$};
	\draw [dashed] (-4.5,1.7) -- (4.5,1.7) node[anchor=west] {$\mathrm{i} \Theta   $};
	\fill [opacity=0.1] (-3.5,0) rectangle (-1,1.7); 
	\fill [opacity=0.1] (3.5,0) rectangle (1,1.7); 
\end{scope}
\end{tikzpicture}
\end{center}
Applying   Lemma \ref{lemma:sing_elastic} we may assume that for $| \omega -\Lambda| < \varepsilon$ the function
$(\zeta_1(\cdot ) - \omega)^{-1}$ has exactly one singularity  in each rectangle. By  Lemma \ref{lemma:sing_elastic} 
the  singularities are given by 
$$ 	\xi_{1+ }(\omega) =  - H( - \mathrm{i} \sqrt{\Lambda - \omega})    \qquad \text{and} \qquad 
\xi_{2+ }(\omega) = H(\mathrm{i} \sqrt{\Lambda - \omega}) $$
for some holomorphic function $H$ satisfying $H(0) = \varkappa$ and  $H' (0) = (\zeta_1''(\varkappa)/2)^{-1/2} $. Note that $H( - \eta) = - H(\eta)$.
Since $\xi \mapsto P_1(\xi)$ depends holomorphically on $\xi$  the residue theorem implies 
\begin{align*}
	&\left( \int_{-\varkappa - \varepsilon}^{-\varkappa + \varepsilon} + \int_{\varkappa - \varepsilon}^{\varkappa + \varepsilon} \right)   F_\omega (\xi)  \; \mathrm{d} \xi \\
	&= 
	\left( \int_{\gamma_1}  + \int_{\gamma_2} \right) 
		\frac{\zeta_1 (\xi)}{\zeta_1 (\xi) - \omega} \scal{P_1(\xi) K_0(\xi) \hat g_\ell(\xi) }{K_0(\overline{\xi}) \hat h_\ell (\overline{\xi}) }_{I_+}  \; \mathrm{d} \xi\\ 
	 &\quad  +  2 \pi \mathrm{i} \bigr[  \mathrm{Res}_{\xi =\xi_{1+}(\omega)} + \mathrm{Res}_{\xi =\xi_{2+}(\omega)} \bigr]
 	\left(\frac{\zeta_1(\xi) \cdot \scal{P_1(\xi) K_0(\xi) \hat g_\ell(\xi) }{K_0(\overline{\xi}) \hat h_\ell(\overline{\xi}) }_{I+} }{\zeta_{1 }(\xi) - \omega} \right) \\ & \quad + \mathcal O(1) . 
\end{align*}
Note that 
$$ \int_{\gamma_j} 
	\frac{\zeta_1(\xi)}{\zeta_1 (\xi) - \omega} \scal{P_1(\xi) K_0(\xi) \hat g_\ell(\xi) }{K_0(\overline{\xi}) \hat h_\ell (\overline{\xi}) }_{I_+} = \mathcal O(1) \qquad \text{for} \quad  j=1,2  .  $$
For the residues  we have 
\begin{align*}
	\mathrm{Res}_{\xi =\xi_{1 + }(\omega)} \frac{1}{\zeta_1(\xi) - \omega} &= \mathrm{Res}_{\xi =\xi_{2 +}(\omega)} \frac{1}{\zeta_1(\xi) - \omega}  
		= \left. \frac{1}{\zeta_1'(\xi)} \right|_{\xi = \xi_{2+} (\omega)} \\
		&= \frac{1}{\zeta_1'(  H( \mathrm{i} \sqrt{\Lambda - \omega}))}  = \sum_{k=-1}^\infty c_k (\Lambda - \omega)^{k/2}
\end{align*}
for coefficients $c_k$. The singular term in  the Laurent series  is given by 
$$ c_{-1} = \frac{1}{\mathrm{i} \;  \zeta_1''(\varkappa) \;   H'(0)} = \frac{1}{\mathrm{i} \sqrt{2 \zeta_1''(\varkappa)}}  \neq 0 . $$
Expanding the term 
$  \scal{P_1(\xi) K_0(\xi) \hat g_\ell(\xi) }{K_0(\overline{\xi}) \hat h_\ell(\overline{\xi}) } $
into a Taylor series at $\xi = \pm \varkappa$ we obtain 
\begin{align*}
	\omega \int_\RR F_\omega (\xi)  \; \mathrm{d} \xi &= \frac{2 \pi \cdot \omega^2}{\sqrt{2 \zeta_1 ''(\varkappa)}\cdot \sqrt{\Lambda - \omega}}
	 \sum_{\diamond = \pm \varkappa} \scal{P_0(\diamond ) K_0 (\diamond )
	 \hat g_\ell ({\diamond )}}{K_0 (\diamond ) \hat h_\ell (\diamond )}_{I_+} + \mathcal O(1) \\
	 &= \frac{2 \pi \cdot \Lambda^2}{\sqrt{2 \zeta_1 ''(\varkappa)}\cdot \sqrt{\Lambda - \omega}}
	 \sum_{\diamond = \pm \varkappa} \scal{P_0(\diamond ) K_0 (\diamond )
	 \hat g_\ell ({\diamond )}}{K_0 (\diamond ) \hat h_\ell (\diamond )}_{I_+} + \mathcal O(1) .
\end{align*}
For $\diamond = \pm \varkappa$ we obtain 
$$  \Lambda^2 \scal{P_1(\diamond) K_0( \diamond ) \hat g_\ell (\diamond)}{
	K_0(\diamond ) \hat h(\diamond)}_{I+} 
	= \Lambda^2 \scal{ K_0(\diamond) g_\ell(\diamond) }{\psi_{\diamond}}_{I_+} \scal{\psi_{\diamond}}{ K_0 (\diamond) \hat h_\ell (\diamond)}_{I+} , $$
and 
\begin{align*}
	\Lambda \scal{ K_0(\diamond) g_\ell(\diamond) }{\psi_{\diamond}}_{I_+} 
	&=   \scal{K_0(\diamond) g_\ell(\diamond)}{A_{\varnothing+}^{(2)}(\diamond) \psi_{\diamond}}_{I_+} 
	=  - 2 \overline{\partial_n \psi_{\diamond,2} (0)} \cdot \hat g_{\ell}(\diamond) . 
\end{align*}
Note that 
\begin{align*}
	\hat g_{\ell}(\diamond) \cdot \overline{\hat h_\ell(\diamond)}  &= \frac{1}{2\pi} \left( \int_{\Gamma} \mathrm{e}^{-\mathrm{i}\diamond x } g_\ell (x)  \; \mathrm{d} x \right)\overline{ \left( \int_{\Gamma} \mathrm{e}^{-\mathrm{i} \diamond x } h_\ell(x)  \; \mathrm{d} x \right)} \\
	&= \frac{1}{2\pi} \scal{g_\ell }{\Phi_\pm}_\Gamma \cdot \scal{\Phi_\pm }{h_\ell }_\Gamma = \frac{|\Gamma|}{2 \pi} \scal{P_\pm g_\ell }{h_\ell}_{I_+}  ,
\end{align*}
where $|\Gamma|$ is the Lebesgue measure of $\Gamma$. 
Finally we obtain 
$$ \Lambda^2 \scal{P_1(\diamond) K_0(\diamond ) \hat h_\ell (\diamond)}{
	K_0( \diamond ) \hat h_\ell ( \diamond)}_{I+} = \frac{2|\Gamma| }{\pi} \cdot | \partial_n \psi_{\diamond,2} (0)|^2 \cdot \scal{P_\pm T_\ell g }{T_\ell h}_{I_+} , $$
which proves Theorem \ref{th:asymptotics_q} since $\psi_{\varkappa,2} =  \psi_{- \varkappa,2}$.

\subsection{The proof of Theorem \ref{th:main_2D}}\label{sec:proof_main_th_2D}
From Theorem \ref{th:asymptotics_q} we obtain 
\begin{align}\label{eq:first_asympt_q}
	\mathcal Q(\ell, \omega) = \frac1\ell \mathcal Q_0 +  \frac{4 | \Gamma| \cdot  |\partial_2 \psi_{\varkappa, 2} (0)|^2}{\sqrt{\Lambda -\omega}
		\cdot \sqrt{2 \zeta_1 ''(\varkappa) }}
		 \;  T_\ell^* \bigr( P_+ + P_- \bigr)  T_\ell +   R(\ell,\omega) 
\end{align}
with the following estimate on the remainder  
\begin{align}\label{est:R(ell,omega)}
	  \sup\{ \| R(\ell,\omega) \|_{\mathcal L(L_2(\Gamma))} :  \ell \in (0,\ell_0 ) \; \wedge \; |\omega - \Lambda| < \varepsilon \} < \infty .
\end{align}
\begin{remark}
	Using a similar argumentation as in Theorem \ref{th:asymptotics_q} it follows 
	that for every compact set $K \subseteq \CC \backslash [ \Lambda , \infty)$ there exists $\ell_0  = \ell_0(\Gamma, K)$ such that 
	$$ \mathcal Q(\ell, \omega) = \frac{1}\ell \mathcal Q_0 + \tilde R(\ell, \omega) , $$
	and the remainder satisfies  $$ \sup \{ \| \tilde R(\ell, \omega) \|_{\mathcal L(L_2(\Gamma))} : \omega \in K \; \wedge \ell \in (0, \ell_0 ) \} < \infty . $$
	Recalling that the operator $\mathcal Q_0$ is a invertible, we obtain  
	$$ \mathcal Q(\ell, \omega) = \frac1\ell \mathcal Q_0 \left( I + \ell \mathcal Q_0^{-1} \tilde R(\ell, \omega) \right) . $$
	Choosing $\ell> 0$ sufficiently small implies  that  $\mathcal Q(\ell, \omega)$ is invertible for all $\omega \in K$ and $\ell \in (0, \ell_0)$. In particular, $0$ cannot be an eigenvalue of $\mathcal Q(\ell, \omega)$. As a consequence 
	the discrete eigenvalues of the operator $A_{\Gamma_\ell}^{(2)} $ converge to $\Lambda $ as $\ell \to 0$. 
\end{remark}
Now we use the the symmetry of the problem  with respect to the axis $x_1 = 0$. We set  
\begin{align}
	L_{2,\mathrm{s}}(\Gamma) &:= \{ g \in L_2(\Gamma) : g(x_1) = g(-x_1) \} ,   \\
	L_{2,\mathrm{as}}(\Gamma) &:= \{ g \in L_2(\Gamma ) : g(x_1) = - g(-x_1) \}  
\end{align}
with projections $P_{\mathrm{s}}$ and $P_{\mathrm{as}}$. Recall that $\Gamma = - \Gamma$. 
Since in we are now considering   a mixed problem on  the upper half-strip  there will be  no risk of confusion with the spaces $\mathcal H^{(\mathrm{s})}$ and $\mathcal H^{(\mathrm{as})}$ introduced before.  
A simple  calculation shows that the forms $q(\ell,\omega)$ and $q_0$ decompose as follows
$$   q(\ell, \omega) =  q^{(\mathrm{s})}(\ell, \omega) \oplus q^{(\mathrm{as})}(\ell, \omega) \qquad \text{and} \qquad   q_0 = q_0^{(\mathrm{s})} \oplus q_0^{(\mathrm{as})} , $$
where 
 $q^{(\dagger)}(\ell, \omega)$ and $q_0^{(\dagger)}$ act in $L_{2,\dagger}(\Gamma)$
for $\dagger \in \{\mathrm{s},\mathrm{as}\}$. Thus, 
$$ \mathcal Q(\ell, \omega)  = P_{\mathrm{s}}^* \mathcal Q(\ell, \omega) P_{\mathrm{s}} + P_{\mathrm{as}}^* \mathcal Q(\ell, \omega) P_{\mathrm{as}}  $$
and Theorem \ref{th:asymptotics_q} implies 
\begin{align*}
	 P_{\dagger }^* \mathcal Q(\ell, \omega) P_{\dagger}  &=  \frac1\ell  P_{\dagger}^* \mathcal Q_0  P_{\dagger}  
		- \frac{4  | \Gamma|  \cdot  |\partial_2 \psi_{\varkappa, 2} (0)|^2}{\sqrt{\Lambda -\omega}
		\cdot   \sqrt{2 \zeta_1 ''(\varkappa) }}
		 \;  P_{\dagger}^*  T_\ell^*  \bigr( P_+ + P_- \bigr) T_\ell  P_{\dagger}  + P_{\dagger}^* R(\ell,\omega) P_\dagger   
\end{align*}
for $\dagger \in \{\mathrm{s},\mathrm{as}\}$. A short calculation shows that for fixed $\dagger \in \{ \mathrm{s}, \mathrm{as} \}$
the operators  $P_{\dagger}^*  T_\ell^*  P_+   T_\ell  P_{\dagger}$ and $P_{\dagger}^*  T_\ell^*  P_-    T_\ell  P_{\dagger}$
coincide. Indeed, we have 
\begin{align*}
	 P_{\mathrm{s}}^*  T_\ell^*  P_\pm    T_\ell  P_{\mathrm{s} }   
		= \frac{\ell}{|\Gamma|} \; \scal{\cdot  }{\Phi_\ell^{(\mathrm{s})}}_{\Gamma} \cdot \Phi_\ell^{(\mathrm{s})}  \qquad \text{and} \qquad
	P_{\mathrm{as}}^*  T_\ell^*  P_\pm    T_\ell  P_{\mathrm{as}}  
		= \frac{\ell}{|\Gamma|} \; \scal{\cdot  }{\Phi_\ell^{(\mathrm{as})}}_{\Gamma} \cdot \Phi_\ell^{(\mathrm{as})} , 
\end{align*}
where
\begin{align*}
	\Phi^{(\mathrm{s})}_\ell(x_1) := \cos (\varkappa \ell x_1) \qquad \text{and} \qquad \Phi^{(\mathrm{as})}_\ell(x_1) := \sin (\varkappa \ell x_1) .
\end{align*}
Thus, 
\begin{align*}
	 P_{\dagger }^* \mathcal Q(\ell, \omega) P_{\dagger}  &=  \frac1\ell  P_{\dagger}^* \mathcal Q_0  P_{\dagger}  
		- \frac{8  \ell \cdot  |\partial_2 \psi_{\varkappa, 2} (0)|^2}{\sqrt{\Lambda -\omega}
		\cdot   \sqrt{2 \zeta_1 ''(\varkappa) }} \; \scal{\cdot  }{\Phi_\ell^{(\dagger)}}_{\Gamma} \cdot \Phi_\ell^{(\dagger)}  
	 + P_{\dagger}^* R(\ell,\omega) P_\dagger  . 
\end{align*}
In particular $P_{\dagger}^*  T_\ell^*  (P_+ + P_-)    T_\ell  P_{\dagger} $ are rank-one operators for $\dagger \in \{\mathrm{s},\mathrm{as}\}$. 
To prove Theorem \ref{th:main_2D} we consider for real $\omega$ not only the kernel of the operator $\mathcal Q(\ell, \omega)$, 
but more generally the discrete eigenvalues of the self-adjoint realisation of  $P_\mathrm{\dagger}^* \mathcal Q(\ell, \omega) P_\mathrm{\dagger}$ in $L_{2, \dagger}(\Gamma)$ for $\dagger \in \{ \mathrm{s}, \mathrm{as} \}$.
For $\ell> 0$ and $\omega \in \RR  \backslash [\Lambda , \infty)$ we denote these  eigenvalues (counted with multiplicities) by 
 $$\mu_1^{(\dagger)}(\ell, \omega) \le \mu_2^{(\dagger)}(\ell, \omega) \le \ldots . $$  
\begin{lemma}\label{lemma:ev_D_ell_2D}
	Let $\ell_0 > 0$ and $\varepsilon > 0$ be  chosen as in Theorem \ref{th:asymptotics_q}.  For $\dagger  \in \{ \mathrm{s}, \mathrm{as} \}$ the  following assertions hold true:
	\begin{enumerate}
		\item For   $\ell > 0 $ the function $ \mu_1^{(\dagger)} (\ell, \cdot )$ is strictly decreasing in the interval 
		$(- \infty,  \Lambda)$. 
		\item For fixed $\ell \in (0,\ell_0)$ we have $ \mu_1^{(\dagger)} (\ell ,\omega) \to -\infty$ as $\omega \to \Lambda$. 
		\item For fixed $\omega \in ( \Lambda - \varepsilon ,   \Lambda )$ we have $\mu_1^{(\dagger)}(\ell, \omega) \to \infty$ as $\ell \to 0$. 
		\item  There exists $\tilde \ell_0  > 0   $ such that for all $\tilde \ell \in (0, \ell_0)$ and for 
		all $| \omega -  \Lambda| < \varepsilon  $ we have $\mu_2^{(\dagger)}(\ell, \omega) > 0$.
	\end{enumerate} 
\end{lemma}
\begin{proof}
	For the proof of (1) we use the decomposition  of the Dirichlet-to-Neumann operator $D_\omega$ given in Theorem \ref{th:perturb-to-N}. 
	For $\omega_1 ,  \omega_2 < \Lambda $ and $g \in \tilde H^{1/2}_0(\Gamma)$  we have  
	\begin{align*}
		 q (\ell, \omega_1) - q (\ell, \omega_2)[g]  &= -  \omega_1  \scal{ ( I + \omega_1 (A_{\varnothing+}^{(2)} - \omega_1 )^{-1}) K_0 T_\ell g }{K_0 T_\ell  g}_{\Omega_+}\\
			&\quad +  \omega_2  \scal{ ( I + \omega_2 (A_{\varnothing+}^{(2)} - \omega_2 )^{-1}) K_0 T_\ell g }{K_0 T_\ell  g}_{\Omega_+} \\[5pt]
			&= \int_{[\Lambda, \infty)} \left( - \frac{\nu  \omega_1}{\nu - \omega_1} +  \frac{\nu \omega_2}{\nu - \omega_2} \right) \; \mathrm{d} \scal{E(\nu) K_0 T_\ell g}{K_0 T_\ell g}_{\Omega_+} , 
	\end{align*}
	where $E(\nu)$ is the spectral resolution  of  the operator $A_{\varnothing+}^{(2)}$.  A short calculation shows that  the above integrand 
	is strictly positive if  $\omega_1 < \omega_2 < \nu$. Now the first assertion  follows from the min-max principle for self-adjoint operator applied to the form $q^{(\dagger)}(\ell, \omega)$ for $\dagger \in \{ \mathrm{s} , \mathrm{as} \}$. 

	Here and subsequently we fix $\dagger \in \{ \mathrm{s} , \mathrm{as} \}$. To prove assertion (2) we use  Theorem 
	\ref{th:asymptotics_q}  and  the min-max principle for self-adjoint operators. Let  $\ell \in (0, \ell_0)$. For  $|\omega - \Lambda| < \varepsilon$ we have  
	$\mu_1^{(\dagger)}  (\ell, \omega) \le q^{(\dagger)} (\ell, \omega)   [g_0 ]$  
   	for every  $g_0 \in \tilde  H^{1/2}_0(\Gamma) \cap L_{2, \dagger}(\Gamma)$ with $\| g_0 \|_{L_2(\Gamma)} =1$. Choosing   $g_0$ such that $$ \scal{T_\ell^* \bigr( P_+ + P_- \bigr)  T_\ell g_0}{g_0}_{\Gamma} \neq 0 , $$
   	we obtain from Theorem \ref{th:asymptotics_q}
   	\begin{align*}
   		\mu_1^{(\dagger)} (\ell, \omega) & \le \frac1\ell \scal{\mathcal Q_0 g_0}{g_0}_\Gamma - \frac{4 |\Gamma|  \cdot  |\partial_2 \psi_{\varkappa, 2} (0)|^2}{\sqrt{\Lambda -\omega}
		\cdot \sqrt{2\zeta_1 ''(\varkappa) }}
		 \cdot \scal{T_\ell^* \bigr( P_+ + P_- \bigr)  T_\ell g_0 }{g_0}_{\Gamma}     +  C_1 , 
	\end{align*}
	which tends to $-\infty$ as $\omega \to \Lambda $. Here $C_1 := \sup\{ \| R(\ell , \omega)\|_{\mathcal L (L_2(\Gamma))} : \ell \in (0, \ell_0) \; \wedge \; |\omega - \Lambda| < \varepsilon \}$. 
	This proves the second assertion. To deduce (3) we recall that $\mathcal Q_0$ is invertible and we have $q_0[g] = \scal{\mathcal Q_0 g}{g}_\Gamma \ge 0$ for all $g \in \tilde H^{1/2}_0 (\Gamma)$. Moreover, there exists 
	$\mu_* > 0$ such that 	$$ \scal{\mathcal Q_0 g}{g}_\Gamma = q_0[g] \ge  \mu_* \|g\|_{L_2(\Gamma)}^ 2 , \qquad g \in \tilde  H^{1/2}_0(\Gamma)  . $$
	This follows since the spectrum of the self-adjoint realisation of $\mathcal Q_0$ is purely discrete and from the fact that $0$ cannot be an eigenvalue.
	Hence, for fixed  $\omega \in \RR \backslash [\Lambda, \infty)$, $| \omega - \Lambda | < \varepsilon$ we have 
	\begin{align*}
		\mu_1^{(\dagger)} (\ell, \omega) &= \inf\{   q^{(\dagger)}(\ell, \omega)[g] : g \in \tilde H^{1/2}_0(\Gamma) \cap L_{2,\dagger}(\Gamma)\; \wedge \; 
		\| g\|_{L_2(\Gamma)} = 1 \} \ge  \frac{\mu_*}\ell   -  C_1 , 
	\end{align*}
	which tends to $\infty$ sas $\ell \to 0$. This proves (3). Assertion (4) follows if we prove that the form $q^{(\dagger)} (\ell,\omega)$ is positive on  a subset 	of codimension $1$. Choose $ g \in \tilde H^{1/2}_0 (\Gamma) \cap L_{2,\dagger}(\Gamma) $, $\|g\|_{L_2(\Gamma)} = 1$, orthogonal to the function $\Phi_{\ell}^{(\dagger)}$. Then 
	\begin{equation*}
  		 q^{(\dagger)}(\ell , \omega) [g]  = \frac1\ell q_0[g]  +  \scal{R(\ell, \omega)g}{g}_\Gamma   \ge \frac{\mu_*}\ell  - C_1  > 0
	\end{equation*}
	for $0 < \ell < \tilde \ell_0 := \min \{ 1, \mu_*/C_1 \}$ and $|\omega -\Lambda| < \varepsilon$. This concludes the proof of Lemma \ref{lemma:ev_D_ell_2D}.
\end{proof}
\begin{lemma}\label{lemma:uniqueness_2D}
	There exists $\ell_0 = \ell_0(\Gamma) > 0$ such that for all $\ell \in (0, \ell_0)$ the 
	operator $A_{\Gamma_\ell}^{(2)}$ has exactly two eigenvalues  below its essential spectrum $[\Lambda, \infty)$. 
\end{lemma}
\begin{proof}
	The assertion follows if we show for some  $\ell_0 > 0$ that for all $\ell \in (0, \ell_0)$ there exists unique 
	$\lambda_1(\ell), \lambda_2(\ell)  \in (- \infty, \Lambda)$ such that $\mu_1^{(\mathrm{s})}(\ell, \lambda_1(\ell)) = 0 =  \mu_1^{(\mathrm{as})}(\ell, \lambda_2(\ell)) $. 
	Fix $\dagger \in \{ \mathrm{s}, \mathrm{as} \}$ and  let $\varepsilon> 0$ be chosen as in Theorem \ref{th:asymptotics_q} and  Lemma \ref{lemma:ev_D_ell_2D}. Using 
	the remark at the beginning of Section \ref{sec:proof_main_th_2D} we  choose $\ell_0 > 0$ such that 
	$\inf \sigma(A_{\Gamma_\ell}^{(2)}) \ge \Lambda - \varepsilon$ and $\mu_2^{(\dagger)} (\ell, \omega) > 0$ for all $ \ell \in (0,\ell_0)$ and 
	$\omega \in (\Lambda- \varepsilon , \Lambda)$, 
	Let   $\ell \in (0, \ell_0)$. If 
	$\omega$ is chosen such that $\mu_1^{(\dagger)} (\ell,\omega) = 0$, then Lemma \ref{lemma:ev_D_ell_2D} (1) implies for $\omega_1 < \omega< \omega_2< \Lambda$   
	$$ \mu_1^{(\dagger)} (\ell,\omega_1) < \mu_1^{(\dagger)} (\ell,\omega) = 0 < \mu_1^{(\dagger)} (\ell,\omega_2)   .  $$
	As a consequence  $A_{\Gamma_\ell}^{(2)}$ may have at most two discrete eigenvalues for $\ell \in (0, \ell_0)$. 
	
	For the sake of completeness we shall also prove the existence of the eigenvalues.  Using Lemma \ref{lemma:ev_D_ell_2D} (3) we may assume that 
	$\mu_1^{(\dagger)} (\ell, \Lambda  - \varepsilon/2) > 0$ for all $\ell \in (0,   \ell_0)$. Fix  $ \ell \in (0,  \ell_0)$. Since $\mu_1^{(\dagger)} (\ell, \omega) \to  - \infty $ as $\omega \to 
	\Lambda$ and $\mu_1^{(\dagger)} (\ell, \omega)$ depends continuously on $\omega$
	it follows that there exists $\tilde \omega  = \tilde \omega(\ell) \in (\Lambda - \varepsilon/2 , \Lambda)$ such that $\mu_1^{(\dagger)} (\ell, \tilde \omega) = 0$. Thus, $\tilde \omega \in \sigma_d (A_{\Gamma_\ell}^{(2)})$. 
\end{proof}
\begin{remark}
	Another  method of proof for  Lemma \ref{lemma:uniqueness_2D} may be  based on a variant  of  operator-valued Rouch\'e's theorem, cf.\ e.g.\ \cite{AmKaLee,GohbergSigal}.
\end{remark}
The proof of  the asymptotic formula for the eigenvalue of $A_{\Gamma_\ell}^{(2)}$ is based on the Birman-Schwinger principle. Using the estimate \eqref{est:R(ell,omega)} we may choose  $\ell_0 > 0$ such  that the operator $ \mathcal Q_0 + \ell R(\ell, \omega) $ is invertible for all $\ell \in (0, \ell_0)$ and  $\omega \in (\Lambda - \varepsilon, \Lambda)$. 
\begin{lemma}[Birman-Schwinger principle for rank-one perturbations]\label{lemma:Birman_Schwinger}
	We denote by   $T : D(T) \subseteq H \to H$ a self-adjoint operator acting in a Hilbert space $H$ with $0 \notin \sigma(T)$. Let 
	$V \in \mathcal L(H)$, $V \ge 0$ be a rank-one operator. Then $0$ is an eigenvalue of $T - V$ if and only if 
	$$ 1 = \mathrm{tr} \left( V^{1/2} T^{-1} V^{1/2} \right) $$
\end{lemma}
Recall that 
\begin{align*}
	\ell  P_{\dagger }^* \mathcal Q(\ell, \omega) P_{\dagger}  &=  P_{\dagger}^* \mathcal Q_0  P_{\dagger}  
		- \frac{8  \ell^2 \cdot  |\partial_2 \psi_{\varkappa, 2} (0)|^2}{\sqrt{\Lambda -\omega}
		\cdot   \sqrt{2 \zeta_1 ''(\varkappa) }} \; \scal{\cdot  }{\Phi_\ell^{(\dagger)}}_{\Gamma} \cdot \Phi_\ell^{(\dagger)}  
	 + \ell P_{\dagger}^* R(\ell,\omega) P_\dagger   
\end{align*}
for $\dagger \in \{ \mathrm{s}, \mathrm{as} \}$. Note that multiplication with $\ell$ does not change the kernel of the corresponding operator. 
To deduce the asmyptotics of the eigenvalue we  apply the Birman-Schwinger principle  with $H = L_{2, \mathrm{\dagger}} (\Gamma)$,
\begin{align*}
	T &:= \mathcal Q_0   + \ell  R(\ell,\omega)  \qquad \text{and} \qquad %\\[8pt]
	V := \frac{8  \ell^2 \cdot  |\partial_2 \psi_{\varkappa, 2} (0)|^2}{\sqrt{\Lambda -\omega}
		\cdot   \sqrt{2 \zeta_1 ''(\varkappa) }} \; \scal{\cdot  }{\Phi_\ell^{(\dagger)}}_{\Gamma} \cdot \Phi_\ell^{(\dagger)}  .
\end{align*}
Then 
$$ V^{1/2} = \sqrt{\frac{8  \ell^2 \cdot  |\partial_2 \psi_{\varkappa, 2} (0)|^2}{\sqrt{\Lambda -\omega}
		\cdot   \sqrt{2 \zeta_1 ''(\varkappa) }}} \cdot \frac{1}{\|\Phi_\ell^{(\dagger)}\|_{L_2(\Gamma)}} 
	  \scal{\cdot }{\Phi^{(\dagger)}_\ell}_\Gamma \;  \Phi^{(\dagger)}_\ell .
$$
Let us now consider the symmetric case. For the choice  $\omega = \lambda_1(\ell)$ the Birman-Schwinger principle implies 
$$ \frac{8   \ell^2 \cdot  |\partial_2 \psi_{\varkappa, 2} (0)|^2}{\sqrt{\Lambda -  \lambda_1(\ell)}
		\cdot \sqrt{2 \zeta_1 ''(\varkappa) }}  \scal{(  \mathcal Q_0  +\ell R(\ell, \lambda_1(\ell)) )^{-1} \Phi_{\ell}^{(\mathrm{s})}}{\Phi_{\ell}^{(\mathrm{s})}}_\Gamma = 1 $$
or equivalently
$$ \sqrt{\Lambda -  \lambda_1(\ell)}  = \frac{8 \ell^2 \cdot  |\partial_2 \psi_{\varkappa, 2} (0)|^2}{\sqrt{ 2 \zeta_1 ''(\varkappa) }}  
	\scal{(  \mathcal Q_0  +\ell R(\ell, \lambda_1(\ell)) )^{-1} \Phi_{\ell}^{(\mathrm{s})}}{\Phi_{\ell}^{(\mathrm{s})}}_\Gamma .  $$
Note that 
\begin{align}
	( \mathcal Q_0  +\ell R(\ell, \omega) )^{-1} &= ( I  + \ell  \mathcal  Q_0^{-1}  R(\ell, \omega))^{-1}  \mathcal  Q_0^{-1}    \notag  \\
		&= \sum_{k=0}^\infty \ell^k \bigr( - \mathcal  Q_0^{-1}  R(\ell, \omega) \bigr)^{k} \mathcal Q_0^{-1}   = \mathcal Q_0^{-1}  + \mathcal O(\ell) , \label{eq:exp_Q+R_ell} 
\end{align}
where the last estimate holds uniformly in $\omega \in (\Lambda, - \varepsilon, \Lambda)$. For sufficiently small  $\ell$ the sum converges absolutely  in $\mathcal L(L_2(\Gamma))$. Using the Taylor expansion
of $\Phi_{\ell}^{(\mathrm{s})}$ we obtain 
$$  \Phi^{(\mathrm{s})}_\ell (x) = \cos (\varkappa \ell x) = 1 + \mathcal O(\ell) = \Psi_{\ct} (x) + \mathcal O(\ell) , $$
where $\Psi_\ct = 1  \in L_{2,\mathrm{s}}(\Gamma)$ is the constant function. Thus,
\begin{align*}
	\sqrt{\Lambda -  \lambda_1(\ell)}  & = \frac{8  \ell^2 \cdot |\partial_2 \psi_{\varkappa, 2} (0)|^2}{\sqrt{2 \zeta_1 ''(\varkappa) }}  
		\scal{(  \mathcal Q_0  +\ell R(\ell, \lambda_1(\ell)) )^{-1} \Phi_{\ell}^{(\mathrm{s})}}{\Phi_{\ell}^{(\mathrm{s})}}_\Gamma  \\
 	&= \frac{8    |\partial_2 \psi_{\varkappa, 2} (0)|^2}{\sqrt{ 2 \zeta_1 ''(\varkappa) }} \cdot 
		\scal{  \mathcal Q_0 ^{-1} \Psi_\ct }{\Psi_\ct}_\Gamma \cdot \ell^2 + \mathcal O(\ell^3) .  
\end{align*}
Setting 
\begin{align}\label{def:nu_1} 
	\nu_1 := \frac{32   |\partial_2 \psi_{\varkappa, 2} (0)|^4}{\zeta_1 ''(\varkappa)} \cdot 
		\scal{  \mathcal Q_0 ^{-1} \Psi_\ct  }{\Psi_\ct }^2_\Gamma  
\end{align}
we obtain 
\begin{align*}
	\Lambda -  \lambda_1(\ell)  & = \nu_1 \cdot \ell^4 + \mathcal O(\ell^5) .  
\end{align*}
It remains to prove that $\nu_1 > 0$. Note that $\zeta_1''(\varkappa) > 0$ and 
$$ \scal{  \mathcal Q_0 ^{-1} \Psi_\ct }{\Psi_\ct}_\Gamma  =  \scal{  \mathcal Q_0 ^{-1/2} \Psi_\ct }{ \mathcal Q_0 ^{-1/2}  \Psi_\ct}_\Gamma  > 0 . $$
Moreover, using the explicit representation of the eigenfunction $\psi_\varkappa$ in  \eqref{eq:eigenfuncion_elast} we obtain  
\begin{align*}
	\partial_2 \psi_{\varkappa, 2} (0) = c_1  \varkappa  \left[ \frac{\Lambda}2 - \varkappa^2\right] \sqrt{ \Lambda-  \varkappa^2} \left[ \cos  \left(\frac\pi2\sqrt{ \frac{\Lambda}{2} -  \varkappa^2} \right) - 
		 \cos \left(\frac\pi2 \sqrt{ \Lambda-  \varkappa^2} \cdot \right)  \right] ,
\end{align*}
where $c_1 \neq 0 $ is a normalising factor. A numerical calculation,  which can be made rigorous by inserting the corresponding power series,
 shows that 
$$  \partial_2 \psi_{\varkappa, 2} (0) \neq 0 . $$
This proves the asymptotic formula for the eigenvalue $\lambda_1(\ell)$. The second eigenvalue is  treated  in the same way. Here  we use the estimate
$$ \Phi^{(\mathrm{as})}_\ell (x) = \sin  (\varkappa \ell x) = \varkappa \ell \cdot x  + \mathcal O(\ell^2 ) =  \varkappa \ell  \cdot \Psi_{\mathrm{id}}(x)  + \mathcal O(\ell^2 ) ,  $$
where $\Psi_{\mathrm{id}}(x) = x$ is the identity function on $\Gamma$. 
As above  we obtain 
\begin{align*}
	\Lambda -  \lambda_2(\ell)  & = \nu_2 \cdot \ell^8 + \mathcal O(\ell^5) 
\end{align*}
where  
\begin{align}\label{def:nu_2} 
	\nu_2 := \frac{32  \cdot \varkappa^4 \cdot     |\partial_2 \psi_{\varkappa, 2} (0)|^4}{\zeta_1 ''(\varkappa)} \cdot 
		\scal{  \mathcal Q_0 ^{-1} \Psi_{\mathrm{id}} }{\Psi_{\mathrm{id}}}^2_\Gamma  > 0 .
\end{align}
For the sake of completeness we want to calculate the expressions $\scal{  \mathcal Q_0 ^{-1} \Psi_\ct }{\Psi_\ct}_\Gamma$ and $\scal{  \mathcal Q_0 ^{-1} \Psi_{\mathrm{id}}}{\Psi_{\mathrm{id}}}_\Gamma$ in the case
of $\Gamma= (-1,1)$. Then the  operator $\mathcal Q_0$ becomes the composition of  the standard finite Hilbert transform and the  derivative.
Using \cite[Formula (4.8)]{AmKaLeeElastic13} or \cite[Section 5.2]{AmKaLee} we obtain
$$ (\mathcal Q_0^{-1} \Psi_\ct )(x) =\sqrt{1- x^2} , $$
which implies
\begin{align*}
	\scal{\mathcal Q_0^{-1} \Psi_\ct }{\Psi_\ct }_{(-1,1)} =  \int_{-1}^1 \sqrt{1-x^2} \; \mathrm{d} x = \frac\pi2 . 
\end{align*}
Moreover, using  \cite[Formula (4.9)]{AmKaLeeElastic13} we obtain 
 	$$ (\mathcal Q_0^{-1} \Psi_{\mathrm{id}})(x)  =  \frac{x}{2} \sqrt{1-x^2} , $$
and thus, 
$$ \scal{\mathcal Q_0^{-1} \Psi_{\mathrm{id}} }{\Psi_{\mathrm{id}}}_{(-1,1)}  = \frac{1}{2} \int_{-1}^1   x^2 \sqrt{1-x^2} \; \mathrm{d} t 
	= \frac{\pi}{16}  . $$
Thus, we obtain 
\begin{align}
	\Lambda -  \lambda_1(\ell)  &= \frac{8   \pi^2      |\partial_2 \psi_{\varkappa, 2} (0)|^4}{\zeta_1 ''(\varkappa)} \cdot \ell^4 + \mathcal O(\ell^5) , \\[6pt]
	\Lambda -  \lambda_2(\ell)  &= \frac{\pi^2  \varkappa^4   |\partial_2 \psi_{\varkappa, 2} (0)|^4}{8 \zeta_1 ''(\varkappa)} \cdot  \ell^8 + \mathcal O(\ell^9) . 
\end{align}
This completes the proof of Theorem \ref{th:main_2D}.

\section{Proof of the main result - 3D}
We recall that in three dimensions we consider only circular cracks, since 
we want to take advantage of  the rotational symmetry of the problem. We put $\Gamma := B(0,1)$,  $\Gamma_\ell := B(0 ,\ell)$ 
and consider the elasticity operator on  $(\RR^2 \times I ) \backslash ( \overline{\Gamma_\ell} \times \{0\})$ with traction-free boundary conditions.
As in the two-dimensional case we reduce the original problem to a problem on the upper half-plate. This  is done  in exactly
the same way, so we shall omit the details. Then for $\omega \in \CC$ the corresponding Poisson problem on the upper half-plate reads as
\begin{equation}
	 ( - \Delta - \mathrm{grad} \; \mathrm{div} ) u = \omega u \qquad \text{in } \RR^2 \times I_+ , 
\end{equation}
and 
\begin{align}
	&\left\{ \begin{array}{rl} \partial_3 u_1 + \partial_1 u_3 &= 0\\[4pt]
	\partial_3 u_2 + \partial_2 u_3 &= 0 \\[4pt]
	- 2 \partial_3 u_3 &= 0   \end{array} \right.  \qquad  \text{on } \RR^2 \times \left\{\frac\pi2\right\} , \\[6pt]
	&\left\{ \begin{array}{rl}
		\partial_3 u_1 + \partial_1 u_3 &= 0  \\[4pt]
		\partial_3 u_2 + \partial_2 u_3 &= 0\\[2pt]
		u_3 &= g  \end{array} \right. \qquad \text{on } \RR^2 \times \{0\} ,
\end{align}
where $g \in H^{1/2}(\RR^2)$ and   $u \in H^1(\Omega_+; \CC^3) \cap \mathcal H_{2+}$. The spaces 
$\mathcal H_{1+}$ and $ \mathcal H_{2+}$ will not be separately introduced in  the three-dimensional case since 
 their definition is obvious.

Applying the  Fourier transform with respect to the first two variables leads to the system 
\begin{align}\label{eq:mixed_FT_3D1}
	\begin{pmatrix}  2 \xi_1^2 + \xi_2^2 - \partial_3^2 & \xi_1 \xi_2 & - \mathrm{i} \xi_1 \partial_3 \\[2pt]
	 \xi_1 \xi_2 	&  \xi_1^2 + 2 \xi_2^2 - \partial_3^2	& - \mathrm{i} \xi_2 \partial_3 \\[2pt]
	- \mathrm{i} \xi_1 \partial_3 & - \mathrm{i} \xi_2 \partial_3 & \xi_1^2  + \xi_2^2 - 2 \partial_3^2 . 
	\end{pmatrix}  \hat u (\xi, x_3 ) = \omega \hat u (\xi, x_3 ) , 
\end{align}
where $(\xi, x_3) \in \RR^2 \times I_+$. Moreover, we have 
%with 
\begin{align}
	&\left\{ \begin{array}{rl}  \partial_3 \hat u_1\left(\xi, \frac\pi2\right)+ \mathrm{i} \xi_1 \hat u_3  \left(\xi, \frac\pi2\right) &= 0 \\[4pt]
	\partial_3 \hat u_2 \left(\xi, \frac\pi2\right)  + \mathrm{i} \xi_2  \hat u_3 \left(\xi, \frac\pi2\right) &= 0 \\[4pt]
	2 \partial_3 \hat u_3 \left(\xi, \frac\pi2\right) &= 0 
	\end{array} \right.  \hspace{0.15cm}  \qquad \quad \text{for } \xi \in \RR^2 , \\[6pt]
	&\hspace{1.05eM} \left\{ \begin{array}{rl} 
		\partial_3 \hat u_1(\xi,0) + \mathrm{i} \xi_1 \hat u_3(\xi, 0) &= 0 \\[4pt]
	\partial_3 \hat u_2(\xi, 0) + \mathrm{i} \xi_2  \hat u_3(\xi,0) &= 0\\[2pt]
		\hat u_3(\xi, 0) &= \hat g(\xi)   \end{array} \right. \qquad  \text{for }  \xi \in \RR^2 . \label{eq:mixed_FT_3D3}
\end{align}
If $\omega \in \CC \backslash [\Lambda, \infty)$ then the Sturm-Liouville problem \eqref{eq:mixed_FT_3D1}-\eqref{eq:mixed_FT_3D3} is uniquely solvable for any $\xi \in  \RR^2$
and we denote by  $\hat u(\xi) = K_\omega(\xi) \hat g(\xi)$ its  unique solution. The rotational
symmetry of the problem implies 
\begin{equation}
	K_\omega(M\xi) = \begin{pmatrix} M & 0 \\ 0 & 1 \end{pmatrix} K_\omega (\xi) 
\end{equation}
for every $M \in \mathrm{SO}(2)$, cf.\ Lemma \ref{lemma:spectral_elast_3D}. In particular, we have to solve the corresponding Poisson
problem only  for $ \xi = (|\xi| , 0) $. For $\omega \in \CC \backslash [\Lambda ,\infty)$ we denote by $K_\omega : H^{1/2}(\RR^2) \to H^1(\Omega_+;\CC^3)  \cap \mathcal H_{2+}$ the Poisson operator and by $D_\omega : H^{1/2}(\RR^2) \to H^{-1/2}(\RR^2)$ the Dirichlet-to-Neumann operator. Then 
\begin{align*}
	(\widehat{K_\omega g}) ( \xi,  \cdot) = K_\omega(\xi) \hat g(\xi) . 
\end{align*}
If $m_\omega$ is given as in previous section, then a  short calculation shows that   the  Dirichlet-to-Neumann operator satisfies 
\begin{equation}
	\widehat{D_\omega g} (\xi) = m_\omega (|\xi|) \hat g(\xi) .  
\end{equation}
%Then 
As in Lemma \ref{lemma:mapping_prop_K_D} we have  the following mapping properties 
$$ K_\omega : H^{s} (\RR^2 ) \to H^{s+1/2}(\Omega_+;\CC^3) , \qquad D_\omega : H^s(\RR^2) \to H^{s-1}(\RR^2) $$
for all  $s\in \RR$. The remaining steps  of the proof are  well known. We define the spaces  
$\tilde H^{1/2}_0(\Gamma_\ell)$ and $H^{-1/2}(\Gamma_\ell)$ as in \eqref{def:Hs_0} and \eqref{def:Hs} and put
\begin{align}
	 D_{\Gamma_\ell, \omega} : \tilde H^{1/2}_0 (\Gamma_\ell) \to H^{-1/2}(\Gamma_\ell ), \qquad D_{\Gamma_\ell, \omega} := r_{\Gamma_\ell} D_{\Gamma_\ell, \omega} e_{\Gamma_\ell} ,
\end{align}
where  $e_{\Gamma_\ell}$ is the extension operator and $r_{\Gamma_\ell}$ is the restriction operator. 
Let
$$ d_{\Gamma_\ell , \omega}[g ,h] := \int_{\RR^2}  m_\omega (|\xi|)  \cdot \hat g(\xi) \; \overline{\hat h(\xi)} \; \mathrm{d} \xi , \qquad 
	g, h \in D[d_{\Gamma_\ell,\omega}] := \tilde H^{1/2}_0(\Gamma_\ell) $$
be the associated sesquilinear form and define  the scaled operator 
\begin{align}
	\mathcal Q (\ell, \omega)  &: \tilde H^{1/2}_0 (\Gamma) \to H^{-1/2}(\Gamma), \qquad  \mathcal Q (\ell, \omega) = \ell \cdot T_\ell^* D_{\Gamma_\ell, \omega}  T_\ell 
\end{align}	
as well as its sesquilinear form 
\begin{align} q(\ell, \omega)[g,h]  =  \int_{\RR^2}   m_\omega (|\xi|/\ell )  \cdot \hat g(\xi) \; \overline{\hat h(\xi)}  \; \mathrm{d} \xi  , \qquad g, h \in D[q(\ell, \omega)] := \tilde H^{1/2}_0(\Gamma) . \end{align} 
Here $ T_\ell : L_2(\Gamma) \to L_2(\Gamma_\ell)$, $ (T_\ell g)( \hat x) := \ell^{-1} g( \hat x /\ell) $. 
As before  we define $\mathcal Q_0 : \tilde H^{1/2}_0(\Gamma) \to H^{-1/2}(\Gamma)$,
\begin{align}  \scal{\mathcal Q_0 g}{h}_\Gamma := q_0 [g,h] := \int_{\RR^2}  |\xi| \cdot \hat g(\xi) \cdot \overline{h(\xi)} \; \mathrm{d} \xi . \end{align}
Using a three-dimensional version of Theorem \ref{th:perturb-to-N} we obtain 
$$ \mathcal Q(\ell, \omega) = \mathcal Q(\ell, 0) -   \omega \cdot T_\ell^* K_0^* (I + \omega (A_{\varnothing +}^{(2)} - \omega)^{-1}) K_0 T_\ell ,  $$
or equivalently
$$ q(\ell, \omega)[g,h] = q(\ell, 0 )[g,h] - \omega \scal{(I + \omega (A_{\varnothing +}^{(2)} - \omega)^{-1}) K_0 T_\ell g}{K_0 T_\ell h}_{\Omega_+}  . $$
Then the estimate $m_0(|\xi|) = |\xi| + \mathcal O(1)$ directly implies  that 
$$ q(\ell,0)[g,h]  = \frac1\ell q_0[g,h] + \mathcal O(1) . $$
Next we give an estimate for the resolvent term. For $\theta \in \RR^2$,  $|\theta| = 1$, we define 
$$ \Phi_{\varkappa \theta } (\hat x)  := \mathrm{e}^{\mathrm{i} \varkappa \theta \cdot  \hat x} , \qquad \hat x \in \RR^2 , $$
and we denote by $P_{\varkappa \theta}$ the projection in $L_2(\Gamma)$ on the subspace spanned by the function $\Phi_{\varkappa \theta}$. 
Moreover, let $\psi_{\varkappa \theta}  = (\psi_{\varkappa \theta,1} , 
 \psi_{\varkappa \theta,2}, \psi_{\varkappa \theta,3})^T $ be chosen such that $\| \psi_{\varkappa \theta} \|_{L_2(I_+;\CC^3)} = 1$
and $$ A_{\varnothing+}^{(2)}(\varkappa \theta) \psi_{\varkappa \theta} = \Lambda \psi_{\varkappa \theta} .   $$
\begin{theorem}
	There exists $\ell_0 > 0$ and $\varepsilon >0$ such that for all $\ell \in (0, \ell_0)$ and $|\omega - \Lambda| < \varepsilon$
	the following expansion holds true 
	\begin{align}
		\mathcal Q(\ell, \omega)  &= \frac1\ell  \mathcal Q_0 - 
			\frac{2  \varkappa \cdot |\partial_3 \psi_{(\varkappa,0),3}(0)|^2  }{\sqrt{\Lambda -\omega} 
		\cdot \sqrt{2 \zeta_1''(\varkappa) }} \cdot   \int_{\{|\theta| = 1\}}
			 T_\ell^*   P_{\varkappa \theta} T_\ell \; \mathrm{d} \theta  + R(\ell,\omega) . 
	\end{align}
	The remainder satisfies the estimate   
	$$  \sup\{ \| R(\ell,\omega) \|_{\mathcal L(L_2(\Gamma))} :  \ell \in (0,\ell_0) \; \wedge \; |\omega - \Lambda| < \varepsilon \} < \infty . $$
\end{theorem}
\begin{proof}
	%Let   $h := T_\ell g$. Then 
	We have 
	\begin{align*}
		\mathcal Q(\ell, \omega)   &=  \frac1\ell \mathcal Q_0    - 
		 \omega \cdot  T_\ell^* K_ 0^* ( I + (A_{\varnothing+}^{(2)} - \omega)^{-1}) K_0 T_\ell +  \mathcal O(1)  \\
		&= \frac1\ell  \mathcal Q_0    -  \omega \cdot  T_\ell^* K_ 0^* (A_{\varnothing+}^{(2)} - \omega)^{-1}) K_0  T_\ell +  \mathcal O(1). 
	\end{align*}
	For $g , h \in \tilde H^{1/2} (\Gamma)$ we put $g_\ell = T_\ell g$ and $h_\ell = T_\ell h$. Then 
	\begin{align*} 
		\scal{T_\ell^* K_ 0^* (A_{\varnothing+}^{(2)} - \omega)^{-1}) K_0 T_\ell  g   }{h}_{\Omega+}   & = \int_{\RR^2} \scal{(A_{\varnothing+}^{(2)}(\xi) - \omega)^{-1} K_0(\xi) \hat g_\ell (\xi) }{ K_0(\xi) \hat h_\ell (\xi)}_{I+} \; \mathrm{d} \xi . 
	\end{align*}
	Introducing polar coordinates and using an  estimate on the resolvent term we obtain as in the two-dimensional case 
	\begin{align*} 
		& \scal{T_\ell^* K_ 0^* (A_{\varnothing+}^{(2)} - \omega)^{-1}) K_0 T_\ell  g   }{h}_{\Omega+} \\ 
		&= \int_{\varkappa -\varepsilon < |\xi| < \varkappa - \varepsilon} 
		 \scal{(A_{\varnothing+}^{(2)}(\xi)  - \omega)^{-1} K_0(\xi) \hat g_\ell  (\xi) }{ K_0(\xi) \hat h_\ell  (\xi)}_{I+} \; \mathrm{d} \xi + \mathcal O(1) \\
		&= \int_{\{\theta = 1 \} } \int_{\varkappa -\varepsilon}^{\varkappa - \varepsilon} 
			\scal{(A_{\varnothing+}^{(2)}(\theta r) - \omega)^{-1} K_0(\theta r) \hat   g_\ell (\theta r)   }
		{ K_0(\theta r) \hat h_\ell (\theta r)}_{I+}    r \; \mathrm{d} r \; \mathrm{d} \theta + \mathcal  O(1) \\
		&= \int_{\{\theta = 1 \} } \int_{\varkappa -\varepsilon}^{\varkappa - \varepsilon}
 			\frac1{\zeta_1(r) - \omega} \scal{P_1(\theta r) K_0(\theta r) \hat   g_\ell (\theta r)   }
			{ K_0(\theta r) \hat h_\ell  (\theta r)}_{I+}   r \; \mathrm{d} r \; \mathrm{d} \theta + \mathcal O(1) . 
	\end{align*}
	We recall that the eigenvalue distribution functions  $\zeta_k(\cdot)$, $k \in \NN$,  are given as in the two-dimensional case. The operator 
	$P_1(\cdot)$ is the projection onto the corresponding eigenspace. Using the residue theorem for the inner integral 
	and a Taylor expansion of the remaining terms we obtain 
	\begin{align*}
		&\int_{\varkappa -\varepsilon}^{\varkappa - \varepsilon}
			\frac1{\zeta_1(r) - \omega} \scal{P_1(\theta r) K_0(\theta r) \hat   g_\ell(\theta r)   }
			{ K_0(\theta r) \hat h_\ell (\theta r)}_{I+} \cdot   r \; \mathrm{d} r  \\
		&= \frac{2\pi \varkappa}{\sqrt{\Lambda - \omega} \cdot \sqrt{2 \zeta_1''(\varkappa)} } 
		\scal{P_1(\theta \varkappa) K_0(\theta \varkappa) \hat   g_\ell (\theta \varkappa)   }
			{ K_0(\theta \varkappa) \hat h_\ell (\theta \varkappa)}_{I+} + \mathcal  O(1) ,
	\end{align*}
	where the reminder may be estimated uniformly in $|\theta| = 1 $.
	Then 
	\begin{align*}	
		& \Lambda^2  \scal{P_1(\theta \varkappa) K_0(\theta \varkappa) \hat   g_\ell(\theta \varkappa)   }{ K_0(\theta \varkappa) \hat h_\ell (\theta \varkappa)}_{I_+} \\
		&\qquad  = \Lambda^2 \scal{ K_0(\varkappa) g_\ell(\varkappa\theta) }{\psi_{\varkappa\theta}}_{I_+} \cdot \scal{\psi_{\varkappa\theta}}{ K_0 (\varkappa \theta) \hat h_\ell (\varkappa\theta)}_{I+} \\
		&\qquad =   |2 \partial_3 \psi_{\theta \varkappa,3}(0)|^2 \cdot 
			\hat g_\ell (\theta \varkappa) \cdot \overline{\hat h_\ell (\theta \varkappa)} . 
	\end{align*}
	We note that $\psi_{\varkappa \theta,3} = \psi_{(\varkappa,0),3}$ does not depend on $\theta $. 
	Moreover, from the particular choice $\Gamma = B(0,1)$ we obtain 
	\begin{align*}
		\hat g_\ell (\theta \varkappa) \cdot \overline{\hat h_\ell (\theta \varkappa)} &= \frac{1}{4\pi^2} \left( \int_{\Gamma} \mathrm{e}^{- \mathrm{i} \varkappa \theta x } g_\ell (x) \; \mathrm{d} x \right)
		 \overline{\left( \int_{\Gamma} \mathrm{e}^{- \mathrm{i} \varkappa \theta x } h_\ell (x) \; \mathrm{d} x \right)} \\
		  & = \frac{1}{4\pi} \scal{P_{\varkappa \theta} T_\ell g}{T_\ell h}_{\Gamma} = \frac{1}{4\pi} \scal{T_\ell^* P_{\varkappa \theta} T_\ell g}{h}_{\Gamma}. 
	\end{align*}
	This  proves the assertion.
\end{proof}
Next we want to use  the  rotational symmetry of the problem. Recall that $\Gamma = B(0,1)$. For  $m \in  \ZZ$ we introduce  the space
\begin{align}   
	L_{2,m} (\Gamma)  := \{ g \in L_2(\Gamma) : g(r \cos \varphi , r\sin  \varphi ) = \mathrm{e}^{\mathrm{i} m  \varphi} \tilde g(r) \;  \text{ for some } \tilde g : (0,1) \to \CC \} 
\end{align}  
with the  corresponding projection
\begin{align}   (P_{m} g )(r \cos \varphi, r \sin \varphi ) = \frac{1}{2\pi} \mathrm e^{\mathrm im\varphi} \int_0^{2\pi} \mathrm e^{-\mathrm im t}
	 g(r\cos t, r \sin t) \; \mathrm{d} t ,  \end{align}
where $(r, \varphi ) \in (0,1) \times (0,2\pi)$. 
Then $$ L_{2}(\Gamma) =  \bigoplus_{m \in \ZZ}   L_{2,m} (\Gamma) . $$
\begin{lemma}\label{lemma:decomp_q_3d}
	The form $q(\ell, \omega)$ admits the following decomposition
	$$ q(\ell,\omega)  := \bigoplus_{m \in \ZZ} q^{(m)}(\ell,\omega) , $$
	where $ q^{(m)}(\ell,\omega)$ acts in the Hilbert spaces  $L_{2,m} (\Gamma)$.
\end{lemma}
\begin{proof} 
	For the proof we use a similar decomposition of the $L_2$-space of functions defined on all of $\RR^2$. We have 
	$$ L_2(\RR^2) = \bigoplus_{m \in \ZZ}   L_{2,m} (\RR^2) , $$
	where $$ L_{2,m} (\RR^2) = \{ g \in L_2(\RR^2) : g(r \cos \varphi , r\sin  \varphi ) = \mathrm{e}^{\mathrm{i} m  \varphi} g(r) \;  \text{ for some } \tilde g :\RR_+ \to \CC  \} . $$
 	Let us denote by $\tilde P_m$ the corresponding projections. A short calculation shows that $\tilde P_m$ commutes with the standard Fourier transform on $\RR^2$. 
	For  $g \in H^{1/2}(\RR^2)$ we have 
	\begin{align*} 
		\| P_m g \|_{H^{1/2}(\RR^2)}^2 &= \|  (1 + |\xi|)^{1/2} \widehat{ ( \tilde P_m g) } (\xi) \|_{L_2(\RR^2_\xi)}^2 = \| (1 + |\xi|)^{1/2} ( \tilde P_m \hat{ g}) (\xi) \|_{L_2(\RR^2_\xi)}^2 \\ 
		&= \| \tilde P_m ((1 + |\xi|)^{1/2} g(\xi) )  \|_{L_2(\RR^2_\xi)}^2 \le \| g \|_{H^{1/2}(\RR^2)}^2 ,  
	\end{align*}
	and thus, $g \in H^{1/2}(\RR^2)$. In the same way we obtain for $g , h \in H^{1/2}(\RR)$ 
	\begin{align*}
		  \scal{\tilde P_m g}{h}_{H^{1/2}(\RR)}  =   \scal{ g}{\tilde P_m h}_{H^{1/2}(\RR)} \quad \text{and} \quad 
		  \scal{\tilde P_{m_1} g}{\tilde P_{m_2} h}_{H^{1/2}(\RR)} = 0  \quad \text{if }  m_1 \neq m_2 . 
	\end{align*}
	Thus, $\tilde P_m$ is a orthogonal projection in $H^{1/2}(\RR^2)$ and 
	$$ H^{1/2}(\RR^2) = \bigoplus_{m \in \ZZ}  H^{1/2}(\RR^2) \cap  L_{2,m}(\RR^2)   . $$ 
 	A similar assertion holds true for $\tilde H^{1/2}_0(\Gamma)$ equipped with the standard scalar product. Now  consider    $\tilde H^{1/2}_0(\Gamma)$ with the scalar product induced by the form $q(\ell, \omega)$. Note that 
 	for $m_1 \neq m_2$ we have 
 	$$  q(\ell, \omega)[P_{m_1} g, P_{m_2} h  ]  = \int_{\RR^2}   m_\omega (|\xi|/\ell ) \cdot (\tilde P_{m_1}  \hat{ g})(\xi) \; \overline{(\tilde P_{m_2}  \hat{ h})(\xi)}  \; \mathrm{d} \xi = 0 .  $$
	As above we obtain
	$$  D[q(\ell, \omega)]  = \bigoplus_{m \in \ZZ}   D[q(\ell, \omega)] \cap L_{2,m}(\Gamma) ,  $$
	which  proves the assertion. 
\end{proof}
\begin{remark}
	Following \cite{HSW} we have  the  decomposition
	$$ \mathcal H_{2+} = \bigoplus_{m \in \ZZ} X_m , $$
	where 
	\begin{align*}
		X_{m} := \Bigr\{ & u \in L^2(\Omega_+;\CC^3) \cap \mathcal H_{2+} : 
		 u ( r\cos \varphi, r \sin
		\varphi , x_3)  =  e^{\mathrm i m \varphi} 
		\begin{pmatrix} M_{\varphi} & 0 \\ 0 & 1 \end{pmatrix} \tilde u (r,x_3 ) \\ & 
			\text{ for some } \tilde u : \RR_+ \times I_+ \to \CC  \biggr\} \notag, 
	\end{align*}
	Here we denote by  $M_{\varphi}  \in \mathrm{SO}(2)$ the planar
	rotation matrix to the rotation angle $\varphi \in (0,2 \pi)$. The elasticity operator with circular crack decomposes as follows 
	$$ A_{\Gamma_\ell+}^{(2)} = \bigoplus_{m \in \ZZ} A_{\Gamma_\ell+}^{(2),m} ,   $$
	where  $A_{\Gamma_\ell+}^{(2),m}$ acts in $X_m$. A short calculations shows that  
	$$ \mathrm{dim}\; \mathrm{ker}\; \left( \mathcal Q_0(\ell, \omega)  \Bigr|_{\tilde H^{1/2}_0(\Gamma) \cap L_{2,m}(\Gamma) } \right) = \mathrm{dim}\; \mathrm{ker}\; \left(A_{\Gamma_{\ell}+}^{(2),m} - \omega \right) . $$
\end{remark}
As a consequence of Lemma \ref{lemma:decomp_q_3d}  we  obtain 
\begin{align*}
	&\mathcal Q(\ell, \omega) 
		= \sum_{m \in \ZZ} P_m^*  \mathcal Q(\ell, \omega)  P_m , 
\end{align*}
where   
\begin{align*}
	& P_m^*  \mathcal Q(\ell, \omega)  P_m \\
	&= \frac1\ell  P_m^*    \mathcal Q_0 P_m  -\frac{2  \varkappa \cdot |\partial_3 \psi_{\varkappa (1,0),3}(0)|^2  }{\sqrt{\Lambda -\omega} 
		\cdot \sqrt{2 \zeta_1''(\varkappa) }} \cdot   \int_{\{|\theta| = 1\}} 
			P_m^*  T_\ell^*   P_{\varkappa \theta} T_\ell P_m \; \mathrm{d} \theta 
	+ P_m^* R(\ell,\omega) P_m .
\end{align*}
Fix $\theta = (\cos \alpha , \sin \alpha)^T$, $\alpha \in (0,2\pi)$. Then for $g, h \in L_2(\Gamma)$ we obtain 
\begin{align*}
	\scal{ P_m^*  T_\ell^*   P_{\varkappa \theta} T_\ell P_m g}{h}_\Gamma 
	&= \frac{1}{\pi} \scal{g}{P_m T_\ell^*  \Phi_{\varkappa \theta} }_\Gamma \cdot \scal{P_m T_\ell^* \Phi_{\varkappa \theta} }{ h}_\Gamma , 
\end{align*}
where we recall that $ \Phi_{\varkappa \theta } (\hat x)  := \mathrm{e}^{\mathrm{i} \varkappa \theta \cdot  \hat x}$, $\hat x \in \RR^2$.
For  $(r, \phi) \in (0,1) \times (0,2\pi)$ we have  
\begin{align*}
	\Bigr( P_m T_{\ell}^* \Phi_{\varkappa  \theta} \Bigr) (r \cos \varphi,  r \sin \varphi)  
 	&=  \frac{\ell  \mathrm{e}^{\mathrm{i} m \varphi }}{2\pi} \int_0^{2\pi}  \mathrm{e}^{-\mathrm{i} m \psi }  
	      \mathrm{e}^{\mathrm{i} \varkappa \ell ( \cos \alpha,  \sin \alpha) \cdot  (r    \cos \psi  , r  \sin \psi )^T }  \; \mathrm{d} \psi  \\
	&=  \ell \cdot  \mathrm{e}^{\mathrm i m \varphi  } \cdot  e^{- \mathrm i m   \alpha  }  \cdot  \frac1{2\pi} \int_{\alpha - \frac{\pi}2 }^{\alpha - \frac{5\pi}{2}} 
		\mathrm{e}^{- \mathrm i r \varkappa \ell \sin t  +  \mathrm{i}m ( t + \frac\pi2) } \; \mathrm{d} t \\
  	&= \ell \cdot   \mathrm{i}^m \cdot  e^{  -\mathrm i m \alpha }  \cdot  \mathrm{e}^{\mathrm i  m \varphi } \cdot    J_m(r \varkappa \ell),  
\end{align*}
where  $J_m$ denotes the Bessel function of the first kind of order $m$, cf.\ \cite[Formula 2.2 (5)]{Watson}. 
Setting 
$$ \Phi_{\ell}^{(m)} (r \cos \varphi , r \sin \varphi ) :=  \mathrm{e}^{\mathrm i m \varphi  }   J_m(r \varkappa \ell) $$
we obtain 
$$ P_m^*  T_\ell^*   P_{\varkappa \theta} T_\ell P_m   = \frac{\ell^2}{\pi}  \; \scal{\cdot}{\Phi_{\ell}^{(m)}}_\Gamma \; \Phi_{\ell}^{(m)}  . $$
In  particular this expression does not depend on $\theta$ any more. 
Thus, we have 
\begin{align*}
	P_m^*  \mathcal Q(\ell, \omega)  P_m  
	&= \frac1\ell P_m^*    \mathcal Q_0  P_m  - 
			\frac{4  \varkappa \cdot \ell^2 \cdot |\partial_3 \psi_{\varkappa (1,0),3}(0)|^2  }{\sqrt{\Lambda -\omega} 
		\cdot \sqrt{2 \zeta_1''(\varkappa) }} \;  \scal{\cdot}{\Phi_{\ell}^{(m)}}_\Gamma \; \Phi_{\ell}^{(m)}   +  P_m^* R(\ell,\omega) P_m . 
\end{align*}
Note that now the singular term is again a rank-one perturbation. As above we may prove that 
for every $m \in \ZZ$ there exists $\ell_0 = \ell_0(m)$ such that  the operator $A_{\Gamma_\ell+}^{(2),m}$ has a unique  eigenvalue $\lambda_m(\ell)$  below its essential spectrum $[\Lambda, \infty)$. Applying the Birman-Schwinger principle we obtain for the eigenvalue
\begin{align*}
	\sqrt{\Lambda - \lambda_m(\ell)}  &=  \ell^3  \cdot \frac{4 \varkappa \cdot |\partial_3  \psi_{(\varkappa,0),3} (0)|^2}{ \sqrt{2 \zeta_1''(\varkappa)} }
		\cdot \scal{(\mathcal Q_0 + \ell R_{\ell,\lambda_m(\ell) }) ^{-1} \Phi_{\ell}^{(m)}}{\Phi_{\ell}^{(m)}}_{\Gamma} \\
	&= \ell^3  \cdot \frac{4 \varkappa \cdot |\partial_3  \psi_{(\varkappa,0),3} (0)|^2}{ \sqrt{2 \zeta_1''(\varkappa)} }
		\cdot \scal{\mathcal Q_0^{-1} \Phi_{\ell}^{(m)}}{\Phi_{\ell}^{(m)}}_{\Gamma} + \mathcal O(\ell^4) .
\end{align*}
For the function $\Phi_{\ell}^{(m)}$ we use its Taylor expansion in the radial direction. For $m \ge 0$ we have  
\begin{align*}
	\Phi_{\ell}^{(m)}(r,  \varphi) &= \mathrm{e}^{\mathrm i m \varphi  }   J_m( r \varkappa \ell) 
	= \frac{\ell^{m} \varkappa^{m} } {2^{m}}  \;  r^{m} \;  \mathrm{e}^{\mathrm{i} m  \varphi} 
	 + \mathcal O(\ell^{m+1}) \\&  = \frac{\ell^{m}  \varkappa^{m}}{2^{m}}  
	\cdot  \Psi_m  (r, \varphi) + \mathcal O(\ell^{m+1}) , 
\end{align*}   
where $ \Psi_m  (r, \varphi) = r^{m} \cdot \mathrm{e}^{\mathrm{i} m \varphi} $. For $m \le 0$ we have
\begin{align*}
	\Phi_{\ell}^{(m)}(r,  \varphi) &= \mathrm{e}^{\mathrm i m \varphi  }  (-1)^m  J_{|m|}( r \varkappa \ell) 
	= (-1)^m \frac{\ell^{|m|}  \varkappa^{|m|}}{2^{|m|}}  
	\cdot  \Psi_m  (r, \varphi) + \mathcal O(\ell^{|m|+1}) , 
\end{align*}  
where we put $\Psi_m  (r, \phi) = r^{|m|} \cdot \mathrm{e}^{\mathrm{i} m \phi}$. Finally,  we have 
\begin{equation}
		\Lambda -\lambda_m (\ell, m) = \rho_m  \cdot \ell^{6+4|m|}  + \mathcal O(\ell^{7+4|m|} ) \qquad \text{as} \quad \ell \to 0 , 
\end{equation} 
where 
\begin{align} \label{def:nu_m}
		\rho_m :=   \ell^{6+4|m|} \cdot  \frac{8 \varkappa^{4|m|+2 } \cdot  |\partial_3  \psi_{(\varkappa,0),3} (0)|^4 }{2^{4|m|} 
	\cdot \zeta_1''(\varkappa)} \cdot \scal{\mathcal Q_0^{-1} \Psi_m  }{ \Psi_m }_{\Gamma}^2   .
\end{align}
We note that $\rho_m > 0$, which  follows from 
$$  \scal{\mathcal Q_0^{-1} \Psi_m  }{ \Psi_m }_{\Gamma} =\scal{\mathcal Q_0^{-1/2} \Psi_m  }{\mathcal Q_0^{-1/2} \Psi_m }_{\Gamma}   > 0  $$
and from $  \partial_3 \psi_{(\varkappa,0),3} (0)  \neq 0$. This proves Theorem \ref{th:main_3D}.

\section*{\bf Acknowledgements:}
The research and in particular A.H. were supported by DFG grant WE-1964/4-1.

\end{document}